\documentclass[a4paper,12pt]{amsart}
\usepackage[a4paper,tmargin=40mm,bmargin=35mm,hmargin=20mm]{geometry}
\usepackage[utf8]{inputenc}
\usepackage[T1]{fontenc}
\usepackage[dvipsnames]{xcolor}
\usepackage{relsize}
\usepackage{amsmath}
\usepackage{amssymb}
\usepackage{amsthm}
\usepackage{amscd}
\usepackage{thmtools}
\usepackage{mathtools}
\usepackage[osf]{Baskervaldx}
\usepackage[baskervaldx]{newtxmath}
\usepackage[cal=boondoxo]{mathalfa}

\usepackage{mathdots}
\usepackage[colorlinks=true,linkcolor=BrickRed,citecolor=Green,urlcolor=Blue]{hyperref}
\usepackage[capitalise]{cleveref}
\usepackage{enumitem}
\usepackage{mathrsfs}  
\usepackage{microtype}
\usepackage{tikz-cd}
\usepackage{adjustbox}
\usepackage{contour}
\usepackage{enumitem}
\usepackage{bm}

\crefname{equation}{}{}

\newtheorem{theorem}{Theorem}[section]
\newtheorem{thmx}{Theorem}

\newtheorem{prop}[theorem]{Proposition}
\newtheorem{cor}[theorem]{Corollary}
\newtheorem{lemma}[theorem]{Lemma}

\theoremstyle{definition}

\newtheorem{rmk}[theorem]{Remark}
\newtheorem{rmks}[theorem]{Remarks}

\newtheorem{ex}[theorem]{Example}
\newtheorem{convention}[theorem]{Convention}

\DeclareMathOperator{\Hom}{\mathrm{Hom}}

\DeclareMathOperator{\Tor}{\mathrm{Tor}}

\DeclareMathOperator{\RHom}{\mathbf{R}\mathrm{Hom}}
\DeclareMathOperator{\RQ}{\mathbf{R}Q}

\newcommand{\Ccal}{\mathcal{C}}
\newcommand{\Dcal}{\mathcal{D}}
\newcommand{\Ecal}{\mathcal{E}}
\newcommand{\Fcal}{\mathcal{F}}
\newcommand{\Gcal}{\mathcal{G}}
\newcommand{\Hcal}{\mathcal{H}}
\newcommand{\Ical}{\mathcal{I}}

\newcommand{\Lcal}{\mathcal{L}}

\newcommand{\Ocal}{\mathcal{O}}
\newcommand{\Pcal}{\mathcal{P}}

\newcommand{\Scal}{\mathcal{S}}
\newcommand{\Tcal}{\mathcal{T}}
\newcommand{\Ucal}{\mathcal{U}}
\newcommand{\Vcal}{\mathcal{V}}
\newcommand{\Wcal}{\mathcal{W}}
\newcommand{\Xcal}{\mathcal{X}}
\newcommand{\Ycal}{\mathcal{Y}}

\newcommand{\Qbb}{\mathbb{Q}}

\newcommand{\Zbb}{\mathbb{Z}}

\newcommand{\1}{\mathbf{1}}

\newcommand{\bc}{\dagger}
\newcommand{\cd}{\dagger}

\makeatletter
\def\pmb@@#1#2#3{\leavevmode\setboxz@h{#3}%
\dimen@-\wdz@
\kern-.5\ex@\copy\z@
\kern\dimen@\kern.25\ex@\raise.4\ex@\copy\z@
\kern\dimen@\kern.2\ex@\raise.3\ex@\copy\z@
\kern\dimen@\kern.15\ex@\raise.2\ex@\copy\z@
\kern\dimen@\kern.25\ex@\box\z@
}
\makeatother

\newcommand{\Bou}[1]{\pmb{\langle} #1 \pmb{\rangle}}
\newcommand{\BouBig}[1]{\pmb{\Big\langle} #1 \pmb{\Big\rangle}}
\newcommand{\SBou}[1]{\pmb{|} #1 \pmb{\rangle}}
\newcommand{\SBouBig}[1]{\pmb{\Big|} #1 \pmb{\Big\rangle}}
\newcommand{\CSBou}[1]{\pmb{|} #1 \pmb{\rangle}_\star}

\newcommand{\SBL}[1]{\mathbf{sBL}(#1)}
\newcommand{\CSBL}[1]{\mathbf{csBL}(#1)}
\newcommand{\BL}[1]{\mathbf{BL}(#1)}

\newcommand{\Perv}[1]{\mathbf{Perv}(#1)}

\newcommand{\WTcal}[1]{\widetilde{\Tcal^{<{#1}}}}
\newcommand{\WDcal}[1]{\widetilde{\Dcal^{<{#1}}}}
\newcommand{\PI}{\Pcal\Ical}

\newcommand{\cl}[1]{\overline{\{{#1}\}}}

\newcommand{\D}{\mathcal{D}}

\newcommand{\cpt}{\mathrm{c}}

\newcommand{\bdd}{\mathrm{b}}

\newcommand{\op}{\mathrm{op}}

\newcommand{\K}{\mathcal{K}}

\newcommand{\Inj}[1]{\mathrm{Inj}(#1)}

\newcommand{\Mod}[1]{\mathrm{Mod}(#1)}

\newcommand{\Flat}{\mathrm{Flat}}

\newcommand{\Qcoh}[1]{\mathrm{Qcoh}(#1)}
\newcommand{\coh}[1]{\mathrm{coh}(#1)}

\newcommand{\depth}{\mathrm{depth}}

\newcommand{\Spec}[1]{\mathrm{Spec}(#1)}
\newcommand{\SH}{\mathcal{SH}}

\newcommand{\Supp}{\mathrm{Supp}}

\newcommand{\supp}{\mathrm{supp}}

\renewcommand*{\Perp}[1]{{}^{\perp_{#1}}}

\newcommand{\hocolim}{\mathrm{hocolim}}

\newcommand{\Loc}{\mathrm{Loc}}

\newcommand{\thick}[1]{\mathrm{thick}(#1)}

\newcommand{\pp}{\mathfrak{p}}

\newcommand{\qq}{\mathfrak{q}}

\newcommand{\mm}{\mathfrak{m}}

\newcommand{\qc}{\mathrm{qc}}

\newcommand{\yo}{\mathbf{y}}

\newcommand{{\tst}}{\textit{t}-}

\DeclareRobustCommand{\scrvar}[1]{%
  \text{\textls[-100]{\usefont{U}{BOONDOX-calo}{m}{n}#1}}%
}
\newcommand{\SHom}{\scrvar{Hom}}
\newcommand{\SHomqc}{\scrvar{Hom}^{\qc}}
\newcommand{\RSHom}{\mathbf{R}\SHom}

\newcommand{\newterm}[1]{\textit{#1}}

\title{Semi-Bousfield classes and nonmonotone perversities}
\author{Dolors Herbera}
\address[D. Herbera]{Departament de Matemàtiques, Universitat Autònoma de Barcelona, 08193 Bellaterra (Barcelona), Spain}
\email{dolors.herbera@uab.cat}

\author{Michal Hrbek}
\address[M. Hrbek]{Institute of Mathematics of the Czech Academy of Sciences, \v{Z}itn\'{a} 25, 115 67 Prague, Czech Republic}
\email{hrbek@math.cas.cz}

\author{Giovanna Le Gros}
\address[G. Le Gros]{Institute of Mathematics of the Czech Academy of Sciences, \v{Z}itn\'{a} 25, 115 67 Prague, Czech Republic}
\email{legros@math.cas.cz}

\subjclass[2020]{}

\thanks{Herbera  was partially supported by the projects MIMECO  PID2023-147110NB-I00 financed by the Spanish Government.} 

\thanks{Hrbek and Le Gros were supported by the project LQ100192601 Lumina
quaeruntur, funded by the Czech Academy of Sciences (RVO 67985840)}


\begin{document}

\begin{abstract}
   In the generality of a rigidly-compactly generated tensor triangulated category, we introduce semi-Bousfield classes in terms of the vanishing of the tensor product in positive degrees with respect to a fixed reasonable $t$-structure. We show that semi-Bousfield classes provide a common generalisation of Bousfield classes and compactly generated tensor-compatible $t$-structures. Then we specialise to the setting of the unbounded derived category $\D_{\qc}(X)$ of a Noetherian scheme $X$ and show that the stratification bijection naturally extends to an assignment which takes a (not necessarily monotone) perversity on $X$ to a semi-Bousfield class in $\D_{\qc}(X)$. If $X$ is regular, this assignment constitutes a stratification of the whole semi-Bousfield lattice, while in the singular case its image consists precisely of those semi-Bousfield classes arising from objects of finite Tor-dimension. Restricting this bijection to monotone perversities recovers the recent classification of compactly generated tensor-compatible $t$-structures of Dubey and Sahoo \cite{DS23}.
\end{abstract}
\maketitle
\tableofcontents
\section*{Introduction}
In the seminal work \cite{BBD82}, Be\u{\i}linson, Bernstein, Deligne, and Gabber showed that integer-valued functions on the set of strata of a suitable stratified topological space $X$, called the \textit{perversities}, give rise to $t$-structures in the derived category of sheaves on $X$. These in turn provide a  convenient categorical framework for cohomology theories, in this case, they were used to conceptualise and generalise the notions of intersection cohomology and perverse sheaves of Goresky and MacPherson. 

The parallel story for coherent sheaves over a Noetherian scheme $X$ was developed by Deligne, although it remained publicly unavailable until the preprint of Bezrukavnikov \cite{Bez00}. In this setting, perversities are integer-valued functions on the set of points of $X$, and the $t$-structures in the bounded derived category $\D^\bdd(\coh{X})$ are induced by those perversities which are both monotone and comonotone. The caveat is that, for this to actually work, $X$ needs to satisfy a weak, yet nonvacuous, condition of admitting a dualising complex, or more generally, being CM-excellent. In fact, based on new results from \cite{HM24} one can produce a monotone and comonotone perversity over certain pathological affine Noetherian schemes which does not induce a $t$-structure in the bounded derived category, in particular providing a counterexample to \cite[Conjecture 7.10]{Sta10}, see the discussion in \cref{ss:comonotone}.

As is often the case, the situation becomes more transparent if we enlarge the scope and work in the setting of the unbounded derived category $\D_{\qc}(X)$ of complexes of sheaves with quasi-coherent cohomology. Here, the collection of all $t$-structures often does not even form a set---see Stanley \cite[\S 8]{Sta10}---so we need to impose a suitable finiteness condition on the $t$-structures we study. Since the $t$-structures on $\D^\bdd(\coh{X})$ arising from perversities always extend to $t$-structures in $\D_{\qc}(X)$ which are generated by perfect complexes, it is natural to study compactly generated $t$-structures. It was proved first in the affine case by Alonso, Jeremías, and Saorín \cite{AJS10} that the compactly generated $t$-structures correspond bijectively to monotone perversities (which are now allowed to also take values in the limit points $-\infty$ and $\infty$), see also the exposition in \cite[\S 2]{HNS24}. The extension of this classification to general Noetherian schemes was recently obtained by Dubey and Sahoo \cite{DS23}, in which monotone perversities correspond bijectively to compactly generated $t$-structures satisfying a tensor compatibility condition (which is vacuous in the affine case); the connection to the bounded derived category setting is summarised in \cite{Sah24}. Also, recently Lank \cite{Lan25} extended this classification to the stacky setting.

Due to the groundbreaking work of Neeman \cite{Nee92}, Alonso, Jeremías, and Souto \cite{AJS04}, and Stevenson \cite{Ste13}, we now know that $\D_{\qc}(X)$ is \newterm{stratified} in the modern tensor-triangular sense of Barthel, Heard, and Sanders \cite{BHS23},  based on the seminal concept  of Benson, Iyengar, and Krause \cite{BIK08}. In effect, this means that the assignment sending an arbitrary subset $Y$ of $X$ to the subcategory $\supp^{-1}(Y)$ of all objects whose Balmer-Favi support is inside $Y$ covers all of the localising $\otimes$-ideals in $\D_{\qc}(X)$. In particular, any localising $\otimes$-ideal is in fact a \newterm{Bousfield class} of a suitable object given as the kernel of the derived tensor product, a classical notion with origins in stable homotopy theory \cite{Bou79}. In fact, stratification always induces a complete lattice isomorphism between the Boolean lattice of subsets of $X$ and the Bousfield lattice of $\D_{\qc}(X)$, see \cite[Theorem~8.8]{BHS23}. The stratification is verified by checking two conditions one can impose on a general big tt-category---the Local to Global principle which allows to reduce the statement to the subcategory of objects supported on a singleton, and the Minimality principle which asserts that the latter subcategory contains no nontrivial localising ideals. Restricting this analysis to  Bousfield classes, the Minimality principle boils down to a Tensor nonvanishing property (also known as the Tensor theorem), see the discussion in \cref{ss:Bousfield-stratification}.

Our starting point is the observation that the two structural classifications we laid out above, can be viewed as a single assignment, which sends a perversity $p$ to the subcategory consisting of objects whose graded support is bounded below by $p$. Equivalently, from the commutative algebra perspective, these are objects whose local depth at each point $x$ is greater than $p(x)$. Restricting this assignment to monotone perversities yields precisely the coaisles of compactly generated tensor-compatible $t$-structures. On the other hand, restricting to perversities which only take values in $\{-\infty,\infty\}$ recovers the stratification assignment itself. The ``pullback'' of these two is easy to interpret as well: A monotone perversity $p$ with values in $\{-\infty,\infty\}$ correspond to a specialisation-closed subset $V=\{x \in X \colon p(x) = \infty\}$ of $X$, the corresponding ideal is then the image of the smashing localisation consisting precisely of $V$-local objects, equivalently the coaisle of the stable $t$-structure generated by the perfect complexes supported on $V$. We thus arrive at the following chart:
\begingroup
\renewcommand*{\arraystretch}{0.5}
\setlength\arraycolsep{1pt}
$$
\begin{tikzcd}[column sep = 0.1em, row sep = 0.1em]
  \begin{Bmatrix} \begin{matrix}   \text{ Smashing} \\ \text{localisations} \end{matrix} & \text{\huge{:}} &  \begin{matrix}  \text{ Monotone perversities} \\ \text{valued in $\{-\infty,\infty\}$} \end{matrix} \end{Bmatrix}  & \subseteq & \begin{Bmatrix} \begin{matrix}   \text{ Bousfield} \\ \text{classes} \end{matrix} & \text{\huge{:}} &  \begin{matrix}  \text{ Perversities valued} \\ \text{ in $\{-\infty,\infty\}$} \end{matrix} \end{Bmatrix} \\
    \rotatebox[origin=c]{270}{$\subseteq$} & & \rotatebox[origin=c]{270}{$\subseteq$} \\
    \begin{Bmatrix} \begin{matrix}   \text{ Compactly generated} \\ \text{$\otimes$-$t$-structures} \end{matrix} & \text{\huge{:}} &  \begin{matrix}  \text{ Monotone perversities } \\ \text{valued in $\Zbb \cup \{-\infty,\infty\}$} \end{matrix} \end{Bmatrix} & \subseteq & \begin{Bmatrix} \begin{matrix}   \text{\huge{\textbf{?}}}  \end{matrix} & \text{\huge{:}} &  \begin{matrix}  \text{ Perversities valued } \\ \text{ in $\Zbb \cup \{-\infty,\infty\}$}  \end{matrix} \end{Bmatrix}\\
\end{tikzcd}
$$
\endgroup

A natural task arises to search for a suitable kind of structure which fits into the ``pushout'' corner of the chart in place of the `\textbf{?}' symbol. To this end, we propose the notion of a \newterm{semi-Bousfield class}. Our definition takes place in the generality of a ``$ttt$-category''---a rigidly-compactly generated $t$ensor $t$riangulated category endowed with a $t$-structure satisfying mild axioms. A semi-Bousfield class is then a refinement of Bousfield classes, where the actual vanishing of the tensor product is replaced by vanishing of its nonpositive cohomology. In this generality, we show that semi-Bousfield classes provide a common generalisation of Bousfield classes one hand and compactly generated $\otimes$-$t$-structures on the other. As in the classical theory, semi-Bousfield classes form a set-sized complete lattice, into which the Bousfield lattice embeds. From another point of view, a semi-Bousfield class can be realised as a nullity class of suspension-closed collection of pure-injective objects, which under mild assumptions fits into a co-$t$-structure. This also naturally leads to the definition of a cohomological semi-Bousfield class.

In what follows, we focus on the case study of the derived category $\D_{\qc}(X)$ of a Noetherian scheme $X$. We define a local depth assignment $\Phi$ and interpret it as an injection from general perversities to semi-Bousfield classes. The neatest possible outcome of this map being a bijection happens precisely if $X$ is a regular scheme, which provides a semi-Bousfield version of stratification for $\D_{\qc}(X)$. This is done by establishing a suitable version of the Local to Global principle and a $t$-structure substitute of the Minimality principle for semi-Bousfield classes. Explicitly, we prove the following theorem. 

\begin{thmx}(\cref{T:regular})
  \label{IT:regular}
  Let $X$ be a Noetherian scheme. Then $X$ is regular if and only if the local depth assignment $\Phi$ induces a lattice isomorphism:
  $$
  \begin{tikzcd}[column sep = 5.0em, row sep = 0.1em]
   \begin{Bmatrix} 
    \text{Perversities on $X$} \end{Bmatrix} \arrow[leftrightarrow]{r}{1-1} & \begin{Bmatrix} 
      \text{Semi-Bousfield classes in $\D_{\qc}(X)$}\end{Bmatrix} \\
  \end{tikzcd}
  $$
  In addition, any cohomological semi-Bousfield class (of a pure-injective object) is a semi-Bousfield class.
\end{thmx}

Note that this provides a qualitative difference between the classical Bousfield lattice and the larger semi-Bousfield lattice---unlike the latter, the former is oblivious to singularities of $X$. The point we wish to emphasise is that the semi-Bousfield lattice actually recognises when $X$ is regular or not.  For singular schemes, the image of $\Phi$ does not span the whole semi-Bousfield lattice, however, we can still describe its image in a very satisfactory way, as formulated in the following main theorem of the paper. For the unexplained terminology, see \cref{ss:main-result}.
\begin{thmx}(\cref{T:main}, \cref{C:main-finite-tor})
  \label{IT:main}
  Let $X$ be a Noetherian scheme. The local depth assignment $\Phi$ induces a lattice isomorphism:
  $$
  \begin{tikzcd}[column sep = 5.0em, row sep = 0.1em]
   \begin{Bmatrix} 
    \text{Perversities on $X$} \end{Bmatrix} \arrow[leftrightarrow]{r}{1-1} & \begin{Bmatrix} 
      \text{Semi-Bousfield classes of} \\ \text{objects of finite Tor-dimension}\end{Bmatrix} \\
  \end{tikzcd}
  $$
\end{thmx}

In the last two sections of the paper, we provide some consequences of \cref{IT:regular} and \cref{IT:main}. First, we show in \cref{s:tstructures} that \cref{IT:main} not only restricts to, but can be used to directly recover the classification of compactly generated $\otimes$-$t$-structures of \cite{DS23} with relatively low effort. This indicates a possible top-down approach to $t$-structure classification in ttt-categories. In \cref{s:Tor-pair}, we explain how the classification can be restricted to $\Qcoh{X}$, the Grothendieck category of quasi-coherent sheaves over $X$, yielding a nonaffine generalisation of the recent Tor-pair classification \cite[Theorem 4.17]{HHLG24}.
\begin{thmx}(\cref{T:Tor-pair})
  Let $X$ be a semi-separated Noetherian scheme. The local depth assignment $\Phi$ induces a lattice isomorphism:
  $$
  \begin{tikzcd}[column sep = 5.0em, row sep = 0.1em]
   \begin{Bmatrix} 
    \text{Perversities $p$ on $X$ such that} \\ \text{ $0 \leq p(x) \leq \depth(\Ocal_{X,x})$ } \end{Bmatrix} \arrow[leftrightarrow]{r}{1-1} & \begin{Bmatrix} 
      \text{Hereditary Tor-pairs in $\Qcoh{X}$ generated } \\ \text{by sheaves of finite flat dimension}\end{Bmatrix} \\
  \end{tikzcd}
  $$
  Furthermore, the image of the assignment contains all hereditary Tor-pairs in $\Qcoh{X}$ if and only if $X$ is regular.
\end{thmx}
The paper is concluded with a couple of appendices. In Appendix~\ref{sec:appendix-vnr} we discuss a non-Noetherian setting, specifically the one of a von Neumann regular ring. In this case, the derived category can fail to be stratified in general. Nevertheless, the semi-Bousfield version of stratification holds and can be again used to recover the classification of compactly generated $t$-structures in terms of perversities, which are continuous with respect to a suitable topology. Finally, in Appendix~\ref{ss:cohomology-nonvanishing} we recall cohomological formulas for grade-sensitive tensor nonvanishing over a local ring, as well as provide a certain generalisation for unbounded complexes, which is required to prove \cref{IT:main} in its full generality.
\subsection*{Acknowledgements} We are grateful to Pat Lank, Leonid Positselski, Charalampos Verasdanis, Jordan Williamson, and Alexandra Zvonareva for useful discussions in various stages of development of the manuscript.
\section{Semi-Bousfield classes in big ttt-categories}
\label{s:semibousfield}
In this section, our aim is to introduce the notion of a semi-Bousfield class in the general setting of a ``big ttt-category'' --- a big tensor triangulated category endowed with a reasonable $t$-structure and to establish some of its basic properties. In the whole paper, by a ``subcategory'' we always mean a full and isomorphism-closed subcategory.
\subsection{Big tt-categories}\cite{Bal05, BF11} Let $\Tcal$ be a \newterm{big tt-category}, which is a shorthand for a \textit{rigidly-compactly generated tensor triangulated category}. This amounts to the following assumptions:
\begin{itemize}
  \item $\Tcal$ is a \newterm{compactly generated triangulated category}. 
  
  \noindent
  That is, $\Tcal$ is triangulated with suspension functor $\Sigma$ admitting all small coproducts such that the subcategory $\Tcal^\cpt$ of compact objects is skeletally small and generates $\Tcal$ in the sense $(\Tcal^\cpt)\Perp{0}=0$, see  Section~\ref{sub:tstructure} for notation.
  \item $\Tcal$ is a \newterm{tensor triangulated category}. 
  
  \noindent
  That is, $\Tcal$ is equipped with a symmetric monoidal product $- \otimes -\colon \Tcal \times \Tcal \to \Tcal$ which is exact, coproduct-preserving, and admits a tensor unit object $\1$ which belongs to $\Tcal^\cpt$. By Brown representability, $- \otimes -$ is closed, that is, for each $X \in \Tcal$ the functor $X \otimes -$ admits a right adjoint $[X,-]\colon \Tcal \to \Tcal$. This defines the internal Hom functor $[-,-]\colon \Tcal^{\op} \times \Tcal \to \Tcal$.
  \item In $\Tcal$, \newterm{rigid} and compact objects coincide.
  
  \noindent
  That is, $\Tcal^\cpt$ coincides with the subcategory consisting of those objects $X \in \Tcal$ such that the natural map $[X,\1] \otimes Y \to [X,Y]$ is an isomorphism for all $Y \in \Tcal$. It follows that $\Tcal^\cpt$ is a tensor triangulated subcategory of $\Tcal$. Indeed, for any $x \in \Tcal^\cpt$ the functor $[x,-]$ is isomorphic to $[x,\1] \otimes -$, thus preserves coproducts, and so $x \otimes -$ preserves compacts. The functor $(-)^* \coloneqq [-,\1]\colon (\Tcal^\cpt)^\op \xrightarrow{\simeq} \Tcal^\cpt$ is then easily verified to be an equivalence.
\end{itemize}
\begin{convention}
    The distinction between positive and negative suspensions play an important role in this paper. Thus  note that  when we say a subcategory is closed under $\Sigma$, or suspensions, we mean only nonnegative suspensions. Whereas, a subcategory is closed under $\Sigma^{-1}$, or desuspensions, we mean only nonpositive suspensions.
\end{convention}
\subsection{$t$-structures}\cite{BBD82} \label{sub:tstructure}
Given two subcategories $\Ccal$ and $\Dcal$ of $\Tcal$, let 
$$\Ccal \star \Dcal = \{X \in \Tcal \colon \exists C \to X \to D \rightsquigarrow, C \in \Ccal, D \in \Dcal\}$$
denote the subcategory of $\Tcal$ of extensions of objects of $\Ccal$ by objects of $\Dcal$. Given a subcategory $\Ccal$ of $\Dcal$, we shall use the following notation for various orthogonal subcategories:
$$\Ccal\Perp{0} = \{X \in \Tcal \colon \Hom_\Tcal(C,X) = 0 ~\forall C \in \Ccal\},$$
$$\Perp{0}\Ccal = \{X \in \Tcal \colon \Hom_\Tcal(X,C) = 0 ~\forall C \in \Ccal\},$$
$$\Ccal\Perp{\leq 0} = \{X \in \Tcal \colon \Hom_\Tcal(C,\Sigma^i X) = 0 ~\forall C \in \Ccal ~\forall i \leq 0\},$$
$$\Perp{\leq 0}\Ccal = \{X \in \Tcal \colon \Hom_\Tcal(X,\Sigma^iC) = 0 ~\forall C \in \Ccal ~\forall i \leq 0\},$$
$$\Ccal\Perp{\Zbb} = \{X \in \Tcal \colon \Hom_\Tcal(C,\Sigma^i X) = 0 ~\forall C \in \Ccal ~\forall i \in \Zbb\},$$
$$\Perp{\Zbb}\Ccal = \{X \in \Tcal \colon \Hom_\Tcal(X,\Sigma^iC) = 0 ~\forall C \in \Ccal ~\forall i \in \Zbb\}.$$
If $\Ccal = \{Y\}$ is a singleton subcategory for some object $Y \in \Tcal$, we write just $Y\Perp{0}$, \textit{et cetera}.

A \newterm{$t$-structure} on $\Tcal$ is a pair $(\Ucal,\Vcal)$ of subcategories such that the following axioms hold:
\begin{enumerate}
  \item[(T1)] $\Vcal \subseteq \Ucal\Perp{0}$,
  \item[(T2)] $\Ucal$ is closed under $\Sigma$,
  \item[(T3)] $\Tcal = \Ucal \star \Vcal$.
\end{enumerate}

As per the usual convention, we call $\Ucal$ the \newterm{aisle} and $\Vcal$ the \newterm{coaisle} of the $t$-structure. 

\subsection{Big ttt-categories} From now on, we assume that $\Tcal$ is a \textit{big ttt-category}, meaning that $\Tcal$ is a tt-category  endowed with a fixed $t$-structure, denoted by $(\Tcal^{\leq 0},\Tcal^{>0})$, and satisfying the following assumptions:
\begin{itemize}
  \item $(\Tcal^{\leq 0},\Tcal^{>0})$ is \newterm{compactly generated}. 
  
  \noindent
  That is, 
  $$\Tcal^{>0} = (\Tcal^{\leq 0,\cpt})\Perp{0} \coloneqq \{X \in \Tcal \colon \Hom_\Tcal(S,X)=0 ~\forall S \in \Tcal^{\leq 0,\cpt}\},$$ 
  where $\Tcal^{\leq 0,\cpt} = \Tcal^{\leq 0} \cap \Tcal^\cpt$.
  \item $(\Tcal^{\leq 0},\Tcal^{>0})$ is \newterm{nondegenerate}. 
  
  \noindent
  That is, $\bigcap_{n \in \Zbb}\Tcal^{\leq n} = 0$ and $\bigcap_{n \in \Zbb}\Tcal^{> n} = 0$, where $\Tcal^{\leq n} = \Sigma^{-n}\Tcal^{\leq 0}$ and $\Tcal^{> n} = \Sigma^{-n}\Tcal^{> 0}$. Equivalently, the induced cohomological functor $H^0\colon \Tcal \to \Hcal$ to the heart $\Hcal = \Tcal^{\leq 0} \cap \Tcal^{\geq 0}$ detects vanishing in the sense that for any $X \in \Tcal$ we have $X = 0$ if and only if $H^i (X) = 0$ for all $i \in \Zbb$, where $H^i(-) \coloneqq H^0(\Sigma^i -)$.
  \item $\1 \in \Tcal^{\leq 0}$ and $\Tcal^{\leq 0} \otimes \Tcal^{\leq 0} \subseteq \Tcal^{\leq 0}$. 
\end{itemize}

\vspace{1em}

\begin{center}
\textit{In the rest of this section, $\Tcal$ will denote a big ttt-category.}
\end{center}

We shall use the following notational conventions. For any object $X \in \Tcal$, let:
$$\Sigma^{-\infty} X = 0 \text{ and } \Sigma^{\infty} X = \coprod_{n \in \Zbb}\Sigma^n X.$$
Furthermore, we let: 
$$\Tcal^{>\infty} = 0 \text{ and } \Tcal^{>-\infty} = \Tcal,$$
as well as
$$\Tcal^{+} = \bigcup_{n \in \Zbb}\Tcal^{>n} \text{ and }\Tcal^{-} = \bigcup_{n \in \Zbb}\Tcal^{\leq n}.$$

In some places, it will be convenient to assume that $\Tcal$ is \newterm{algebraic}, that is, $\Tcal$ is equivalent to the stable category of a Frobenius exact category. Recall that a Frobenius exact category is an exact category with enough projective and injective objects, and the projective objects and injective objects coincide. 
\subsection{Tensor $t$-structures revisited}
Following \cite{DS23}, a $t$-structure $(\Ucal,\Vcal)$ in $\Tcal$ is a \newterm{$\otimes$-$t$-structure} if $\Tcal^{\leq 0} \otimes \Ucal \subseteq \Ucal$. Equivalently, one has $[\Tcal^{\leq 0},\Vcal] \subseteq \Vcal$. The following easy but useful observation shows that this condition can be checked in terms of orthogonality of the two constituents of the $t$-structure with respect to the internal Hom functor $[-,-]$.
\begin{lemma}\label{tensor-tstr}
  Let $(\Ucal,\Vcal)$ be a $t$-structure in $\Tcal$. The following are equivalent:
  \begin{enumerate}
    \item[(i)] $(\Ucal,\Vcal)$ is a $\otimes$-$t$-structure,
    \item[(ii)] $[\Ucal,\Vcal] \subseteq \Tcal^{>0}$, that is, for any $U \in \Ucal$ and $V \in \Vcal$ we have $[U,V] \in \Tcal^{>0}$.
  \end{enumerate}
  \begin{proof}
    $(i) \implies (ii)\colon$ For any $X \in \Tcal^{\leq 0}$ we have $\Hom_\Tcal(X,[U,V]) \cong \Hom_\Tcal(X \otimes U,V)$. By $(i)$, $X \otimes U \in \Ucal$ and so the latter $\Hom$ group vanishes. We proved that $\Hom_\Tcal(\Tcal^{\leq 0},[U,V]) = 0$ and therefore $[U,V] \in \Tcal^{>0}$.

    $(ii) \implies (i)\colon$ Let $X \in \Tcal^{\leq 0}$ and let us prove that $X \otimes U \in \Ucal$ for $U \in \Ucal$. For $V \in \Vcal$, we have $\Hom_\Tcal(X \otimes U,V) \cong \Hom_\Tcal(X,[U,V])$, which vanishes because $[U,V] \in \Tcal^{>0}$ by $(ii)$.
  \end{proof}
\end{lemma}
\begin{rmk}
  In view of \cref{tensor-tstr}, we see that $\otimes$-$t$-structures can be defined similarly to $t$-structures with the external Hom suitably replaced by the internal Hom. Indeed, a pair $(\Ucal,\Vcal)$ forms a $\otimes$-$t$-structure if and only if the axioms (T2) and (T3) of the definition of a $t$-structure hold together with the following replacement of (T1):
\begin{enumerate}[leftmargin=2cm]
  \item[($\otimes$-T1)] $[\Ucal,\Vcal] \subseteq \Tcal^{>0}$.
\end{enumerate}
\end{rmk}
\subsection{Definable subcategories}\label{ss:definable}
\cite{Kr00, Bel00} Let $\Mod{\Tcal^\cpt}$ be the right module category over the ringoid $\Tcal^\cpt$, that is, the category of additive functors $(\Tcal^\cpt)^\op \to \Mod{\Zbb}$, and consider the \newterm{restricted Yoneda embedding} $\yo\colon \Tcal \to \Mod{\Tcal^\cpt}$ given by 
$$\yo(X) = \Hom_\Tcal(-,X)_{\restriction \Tcal^\cpt}.$$ 
A triangle $X \xrightarrow{f} Y \xrightarrow{g} Z \xrightarrow{h} \Sigma X$ is \newterm{pure} if the image $0 \to \yo X \xrightarrow{\yo f} \yo Y \xrightarrow{\yo g} \yo Z \xrightarrow{\yo h} 0$ is a short exact sequence. In particular, the map $f$ is sent to a monomorphism $\yo(f)$ and we call such maps \newterm{pure monomorphism} and say that $X$ is a \newterm{pure subobject} of $Y$. Dually, $g$ is a \newterm{pure epimorphism} and $Z$ is a \newterm{pure quotient} of $Y$. The map $h$ is sent via $\yo$ to zero and we call these the \newterm{phantom} maps. Following \cite[Appendix A]{BKS20}, the tensor product $- \otimes -$ extends to a unique colimit-preserving symmetric monoidal product on the Grothendieck category $\Mod{\Tcal^\cpt}$, which we also denote $- \otimes -$, such that $\yo$ is a tensor functor. An object $E \in \Tcal$ is \newterm{pure-injective} if $\yo(E)$ is injective in $\Mod{\Tcal^\cpt}$, or equivalently, if any pure monomorphism with domain $E$ is a split monomorphism. Furthermore, $\yo$ induces an equivalence between the pure-injective objects of $\Tcal$ and the injective objects of $\Mod{\Tcal^\cpt}$.

A subcategory $\Ccal$ of $\Tcal$ is \newterm{definable} if there is a set $F$ of morphisms in $\Tcal^\cpt$ such that $\Ccal = F\Perp{0} = \{X \in \Tcal \colon \Hom_\Tcal(f,X) = 0 ~\forall f \in F\}$. Definable subcategories are closed under products, pure subobjects (and thus also coproducts), and pure quotients. In fact, if $\Tcal$ has a model (i.e., it is the homotopy category of a monoidal model category) then the converse is also true, providing an intrinsic characterisation, see \cite[Proposition 6.8]{BW24}. 
Another useful and equivalent description is that the definable subcategories of $\Tcal$ are precisely those of the form $\yo^{-1}(\Fcal)$, where $\Fcal \subseteq \Mod{\Tcal^\cpt}$ is a torsion-free class of a hereditary torsion pair of finite type.
\subsection{Semi-Bousfield classes}
Recall that given an object $X \in \Tcal$, a \newterm{Bousfield class} of $X$ is the subcategory 
$$\Bou{X} = \{Y \in \Tcal \colon X \otimes Y = 0\}.$$ 
The following generalisation of the notion of Bousfield classes will be the central focus of this paper. Given a subcategory $\Xcal \subseteq \Tcal$, we define the \newterm{semi-Bousfield class} of $\Xcal$ as

$$\SBou{\Xcal} = \{Y \in \Tcal \colon X \otimes Y \in \Tcal^{>0} ~\forall X \in \Xcal\}.$$
In the case $\Xcal = \{X\}$ is a single object, we simply write  $\SBou{X}$ and call the result the \newterm{semi-Bousfield class} of $X$. As we will observe soon in \cref{Bousfield-set}, allowing whole subcategories as opposed to just objects is  a matter of notational convenience, rather than an actual generalisation.
\begin{rmk}
  An incarnation of the notion of a semi-Bousfield class, as well as part of the following Lemma, can be found implicitly in the proof of \cite[Lemma 2.1]{Nee18}.
\end{rmk}
\begin{lemma}\label{SBousfield-closure}
  Let $~\Xcal$ be a subcategory of $~\Tcal$. Then $\SBou{\Xcal}$ is closed under coproducts, $\Sigma^{-1}$, extensions, pure subobjects, and pure quotients. If $~\Tcal$ has a model, then $\SBou{\Xcal}$ is closed under filtered homotopy colimits.
\end{lemma}
\begin{proof}
  The closure under coproducts, $\Sigma^{-1}$, and extensions is straightforward to check. Let $Y' \to Y \to Y'' \to \Sigma Y'$ be a pure triangle and assume that $Y \in \SBou{\Xcal}$. For any $X \in \Xcal$, $Y' \otimes X \to Y \otimes X \to Y'' \otimes X \to \Sigma Y' \otimes X$ is also a pure triangle by \cite[Proposition 2.10(a)]{BKS19}, and thus $Y \otimes X \in \Tcal^{>0}$ implies that both $Y' \otimes X$ and $Y'' \otimes X$ also belong to $\Tcal^{>0}$, as $\Tcal^{>0}$ is a definable subcategory of $~\Tcal$. The final claim follows as a filtered homotopy colimit can be realised as a pure quotient of a coproduct, see the proof of \cite[Lemma 4.5]{Lak20}.
\end{proof}

Iyengar and Krause proved in \cite{IK13} that the collection of Bousfield classes forms a set in a well-generated tensor triangulated category. The same holds in our setting for semi-Bousfield classes.

\begin{lemma}\label{Bousfield-set}
The following hold. 
  \begin{enumerate}
    \item[(i)] Let $~\Xcal$ be a subcategory of $~\Tcal$. Then there is an object $X \in \Tcal$ such that $\SBou{X} = \SBou{\Xcal}$.
    \item[(ii)] The semi-Bousfield classes of $\Tcal$ form a set.
  \end{enumerate}
\end{lemma}
\begin{proof}
  Follows from the more general \cref{cBousfield-set}.
\end{proof}
The right nondegeneracy of our fixed $t$-structure ensures that the notion of a semi-Bousfield class does indeed generalise that of a Bousfield class. 
\begin{lemma}\label{SBou-Bou}
  Any Bousfield class is a semi-Bousfield class. A semi-Bousfield class is a Bousfield class if and only if it is closed under positive suspension $\Sigma$.
\end{lemma}
\begin{proof}
  Let $X \in \Tcal$ and we claim that $\Bou{X} = \SBou{\Sigma^{\infty}X}$, the $\subseteq$ inclusion is clear. For the converse, observe that $Y \in \SBou{\Sigma^{\infty}X}$ if and only if $Y \otimes X \in \Tcal^{>n}$ for all $n \in \Zbb$. This implies $Y \otimes X \in \bigcap_{n \in \Zbb}\Tcal^{>n} =  \Tcal^{>\infty} = 0$ by our assumption of nondegeneracy of the base $t$-structure.

  For the second claim, assume that $\SBou{X}$ is closed under $\Sigma$. Then also $\SBou{\SBou{X}}$ is easily seen to be closed under $\Sigma$, and thus $\SBou{X} = \SBou{\Sigma^{\infty} X} = \Bou{X}$ is a Bousfield class by the previous claim.
\end{proof}
\subsection{Semi-Bousfield lattice and pairs}
In view of the above, the semi-Bousfield classes of $\Tcal$ form a partially ordered set with the order given by $\SBou{X} \leq \SBou{Y}$ if and only if $\SBou{X} \supseteq \SBou{Y}$. This poset is in fact a complete lattice: For any collection $X_i, i \in I$ of objects of $\Tcal$ the join is given as 
$$\bigvee_{i \in I}\SBou{X_i} = \SBouBig{\coprod_{i \in I}X_i}$$ 
while the infinite meet is then obtained formally from infinite joins via 
$$\bigwedge_{i \in I}\SBou{X_i} = \bigvee \{ \SBou{Y} \colon \SBou{Y} \leq \SBou{X_i} ~\forall i \in I\}.$$ 
We shall call this poset a \newterm{semi-Bousfield lattice} of $\Tcal$ and denote it as $\SBL{\Tcal}$.

A \newterm{semi-Bousfield pair} is a pair $(\Xcal,\Ycal)$ of subcategories of $\Tcal$ such that $\Xcal = \SBou{\Ycal}$ and $\Ycal = \SBou{\Xcal}$. Any semi-Bousfield class fits into a semi-Bousfield pair, as for any subcategory $\Xcal$ of $\Tcal$, the pair $(\SBou{\Xcal},\SBou{\SBou{\Xcal}})$ is a semi-Bousfield pair. If one (or equivalently, both) of the semi-Bousfield classes in the semi-Bousfield pair is a Bousfield class, we call the pair a \newterm{Bousfield pair}.

\begin{rmk}\label{pair-lattice}
The order on the semi-Bousfield lattice can be naturally re-interpreted in terms of semi-Bousfield pairs. Consider a poset whose elements are the semi-Bousfield pairs $(\Xcal,\Ycal)$ and the order is given by inclusion of the second component. Then this poset is a complete lattice: Given a set $(\Xcal_i,\Ycal_i), i \in I$ of semi-Bousfield pairs, their join is given by taking intersection $(\bigcap_{i \in I}\Xcal_i,-)$ in the first component while the meet is given by taking intersection $(-,\bigcap_{i \in I}\Ycal_i)$ in the second component. Indeed, this is well-defined: Let $X_i, Y_i \in \Tcal$ be such that $\Xcal_i = \SBou{X_i}$ and $\Ycal_i = \SBou{Y_i}$, then $\bigcap_{i \in I}\Xcal_i = \bigvee_{i \in I}\SBou{X_i}$ and $\bigcap_{i \in I}\Ycal_i =\bigvee_{i \in I}\SBou{Y_i}$ are semi-Bousfield classes. The meet operation also corresponds to the meet in $\SBL{\Tcal}$, because $(\bigwedge_{i\in I}\Xcal_i,\bigvee_{i \in I}\Ycal_i)$ is a semi-Bousfield pair. Indeed, $Z \in \bigwedge_{i \in I}\SBou{X_i}$ translates to $$Z \in \bigcap \{ \SBou{Y} \colon \SBou{X_i} \subseteq \SBou{Y} ~\forall i \in I\},$$
or equivalently, 
$$Z \in \bigcap \{ \SBou{Y} \colon \SBou{Y_i} \supseteq \SBou{\SBou{Y}} ~\forall i \in I\}.$$  
This is further equivalent to $Z \otimes V \in \Tcal^{>0}$ for any $V \in \Tcal$ such that $V \otimes Y_i \in \Tcal^{>0}$ for all $i \in I$, which translates precisely to $Z \in \SBou{\bigvee_{i \in I}\Ycal_i}$.
\end{rmk}

The relationship with the classical Bousfield lattice $\BL{\Tcal}$ is as follows. In general, we have $\Bou{X} = \SBou{\Sigma^{\infty}X} = \bigcap_{n \in \Zbb} \Sigma^{-n}\SBou{X}$. Therefore, the following hold:
$$\SBou{X} \leq \SBou{Y} \implies \Bou{X} \leq \Bou{Y}$$
$$\Bou{X} \leq \Bou{Y} \iff \SBou{\Sigma^{\infty}X} \leq \SBou{\Sigma^{\infty}Y}.$$
In particular, the map $\Bou{X} \mapsto \SBou{\Sigma^{\infty} X}$ is a lattice embedding of the Bousfield lattice into the semi-Bousfield lattice. By Lemma~\cref{SBou-Bou}, its image consists of those semi-Bousfield classes which are closed under $\Sigma$.

\begin{prop}\label{Bou-sublattice}
  The assignment $\BL{\Tcal} \to \SBL{\Tcal}$ given by $\Bou{X} \mapsto \SBou{\Sigma^{\infty} X}$ realises the Bousfield lattice as a complete sublattice of the semi-Bousfield lattice.
\end{prop}
\begin{proof}
  By the discussion above, $\BL{\Tcal}$ is a subposet of $\SBL{\Tcal}$. The join operation is easy to check to agree in this subposet. Let $X_i, i \in I$, be a collection of objects of $\Tcal$. To check that also arbitrary meets agree, it is enough to show that whenever $Y$ is such that $\SBou{Y} \leq \SBou{\Sigma^{\infty}X_i}$ for all $i \in I$ then $\SBou{\Sigma^{\infty} Y} \leq \SBou{\Sigma^{\infty}X_i}$. The assumption translates as: If $Z \in \Tcal$ satisfies $Z \otimes X_i = 0$ for some $i \in I$ then $Z \otimes Y \in \Tcal^{>0}$. However, if the hypothesis holds for $Z$, it clearly holds also for $\Sigma^{\infty}Z$, and then the conclusion is $Z \otimes Y \in \Tcal^{>\infty} = 0$, so that $Z \in \SBou{\Sigma^{\infty}Y} = \Bou{Y}$.
\end{proof}
\begin{rmk}
  A classical property studied in Bousfield lattices is the existence of complements, that is, a pair of objects $X, Y \in \Tcal$ such that $\Bou{X} \wedge \Bou{Y} = \Tcal$ and $\Bou{X} \vee \Bou{Y} = 0$. In such a case, we have a (semi-)Bousfield pair $(\Bou{X},\Bou{Y})$: Indeed, $\Bou{X \otimes Y} \supseteq \Bou{X} \wedge \Bou{Y} = \Tcal$ implies $X \otimes Y = 0$, and thus $\Bou{\Bou{X}} \subseteq \Bou{Y}$. Conversely, let $Z \in \Bou{Y}$ and $V \in \Bou{X}$. Then $\Bou{Z \otimes V} \subseteq \Bou{X} \vee \Bou{Y} = 0$, showing that $Z \otimes V = 0$, and so $Z \in \Bou{\Bou{X}}$.

  However, it is not true that given a Bousfield pair $(\Bou{X},\Bou{Y})$ the objects $X,Y$ are complements. Every Bousfield class fits into a Bousfield pair, while not every Bousfield class admits a complement, e.g. \cite[Corollary 7.4]{DP08}. Indeed, it can happen that $\Bou{X} \vee \Bou{Y} = \Bou{X} \cap \Bou{Y} \neq 0$.
\end{rmk}
\subsection{Compactly generated $\otimes$-$t$-structures}
Recall from \cite{DS23} that a compactly generated $t$-structure $(\Ucal,\Vcal)$ is a $\otimes$-$t$-structure if and only if $\Ucal^\cpt = \Ucal \cap \Tcal^\cpt$ is a $\otimes$-preaisle, that is, it is closed under tensoring by objects from $\Tcal^{\leq 0, \cpt}$. Note that it is defined relative to the fixed $t$-structure.

\begin{lemma}\label{cgtstrtotensor}
  The following are equivalent for a subcategory $\Vcal$ of $\Tcal$ and a subcategory $\Scal$ of $\Tcal^\cpt$:
  \begin{enumerate}    
    \item[(i)] $\Vcal$ is the coaisle of the  $\otimes$-$t$-structure (compactly) generated by $\Scal$,
    \item[(ii)] $\Vcal = \{X \in \Tcal \colon [S,X] \in \Tcal^{>0} ~\forall S \in \Scal\}$,
    \item[(iii)] $\Vcal = \SBou{\Scal ^*}$ where $\Scal^*=\{S^* \colon S \in \Scal\} \subseteq \Tcal^\cpt$,
  \end{enumerate}
\end{lemma}
\begin{proof}
   Assume $(i)$. Let $(\Ucal,\Vcal)$ be the compactly generated $\otimes$-$t$-structure generated by $\Scal$. That is, $\Vcal = \Scal\Perp{\leq 0}$. We are going to prove that $(ii)$ holds, that is, that $\Vcal = \{X \in \Tcal \colon [\Scal,X] \in \Tcal^{>0}\}$. Indeed, one inclusion is \cref{tensor-tstr}. The other follows from the isomorphism 
   $$\Hom_\Tcal(S,\Sigma^i X) \cong \Hom_\Tcal(\Sigma^{-i}\1,[S,X])$$ 
   as, by our assumptions, $\Sigma^{-i}\1 \in \Tcal^{\leq 0}$ for all $i \leq 0$. 
   
   Assume $(ii)$. The duality $(-)^*$ on $\Tcal^c$  yields  $\Vcal = \SBou{\Scal^*}$ which shows that $(ii)$ implies $(iii)$. 
   
   It remains to prove that  $(iii)$ implies $(i)$.  Arguing as above, $\SBou{\Scal ^*} = \{X \in \Tcal \colon [\Scal,X] \in \Tcal^{>0}\}$. It is not hard to check that the latter subcategory coincides with $\widehat{\Scal}\Perp{\leq 0}$, where 
   $$\widehat{\Scal} = \{S \otimes C \colon S \in \Scal, C \in \Tcal^{\leq 0, \cpt}\},$$ 
   and this is a coaisle of a compactly generated $\otimes$-$t$-structure.
\end{proof}

If $\Tcal$ has a model, we call a $t$-structure $(\Ucal,\Vcal)$ \newterm{homotopically smashing} if $\Vcal$ is closed under filtered homotopy colimits.

\begin{lemma}\label{product-tstr}
  Assume that $\Tcal$ has a model. Let $\Xcal$ be a semi-Bousfield class. Then there is a $t$-structure of the form $(\Ucal,\Xcal)$ if and only if $\Xcal$ is closed under products. Such a $t$-structure has a definable coaisle, in particular, it is homotopically smashing.
\end{lemma}
\begin{proof}
  By \cref{SBousfield-closure} and \cref{ss:definable}, the subcategory $\Xcal$ of $\Tcal$ is definable if and only if it is closed under products.
  In \cref{ss:cohomologicalsb}, in particular using \cref{csB-orth} and \cref{tensorstr-t-costr}, we will see that $\Xcal=\Perp{>0}\Ecal$ for some collection of pure-injectives $\Ecal$. Therefore the existence of the $t$-structure then follows from \cite[Lemma 4.8]{AMV17}. 
\end{proof}
\subsection{Brown-Comenetz duality}\label{ss:bw-duality}
The \newterm{character duality} on the category $\Mod{\Zbb}$ of abelian groups is defined as the conservative exact functor
$$(-)^\cd\coloneqq \Hom_\Zbb(-,\Qbb/\Zbb)\colon \Mod{\Zbb}
^\op \to \Mod{\Zbb}.$$ 
Given $X \in \Tcal$, we shall use the same notation $X^\bc \in \Tcal$ for the \newterm{Brown-Comenetz dual} of $X$, that is, a uniquely determined object such that we have the natural isomorphism
$$\Hom_\Tcal(-,X^\bc) \cong \Hom_\Tcal(\1,X \otimes -)^\cd.$$
\begin{rmk}\label{BC-ring}
  Let $R$ be a commutative ring. Then the character duality can also be considered as a triangle functor
  $$(-)^\cd\colon \D(\Mod{R})^\op \to \D(\Mod{R}).$$
  For any $X \in \D(\Mod {R})$, we have the natural isomorphism 
  $$\Hom_{\D(\Mod {R})}(-,X^\cd) \cong \Hom_{\D(\Mod {R})}(R,X \otimes_R^\mathbf{L} -)^\cd,$$
  showing that $(-)^\cd$ is the Brown-Comenetz duality on $\Tcal = \D(\Mod R)$. In particular, there is no conflict of notation here.
\end{rmk}
We gather some well-known basic properties of Brown-Comenetz duality. 
\begin{lemma}\label{BW-duality}
Let $X, Y \in \Tcal$ and $S \in \Tcal^{c}$.
  \begin{enumerate}
    \item There is a natural isomorphism $(X \otimes Y)^\bc \cong [X,Y^\bc]$.
    \item An object of $\Tcal$ is pure-injective if and only if it is a retract in an object of the form $X^\bc$.
    \item There is a natural morphism $X \to X^{\bc\bc}$ to the double Brown-Comenetz dual which is always a pure monomorphism.
    \item There is a natural isomorphism $\Hom_\Tcal(S,X)^\cd \cong \Hom_\Tcal(S^*,X^\bc)$.
  \end{enumerate}
\end{lemma}
\begin{proof}
  \begin{enumerate}
    \item By \cite[Lemma 3.12]{BW24}, there is a natural isomorphism $X^\bc \cong [X,\1^\bc]$. Applying it twice together with the adjunction, we obtain the desired natural isomorphism 
    $$(X \otimes Y)^\bc \cong [X \otimes Y,\1^\bc] \cong [X,[Y,\1^\bc]] \cong [X,Y^\bc].$$
    \item This is \cite[Proposition 3.12, Proposition 3.15]{CS98}.
    \item This is \cite[Proposition 4.13]{CS98}.
    \item This is a special case of \cite[Proposition 6.11]{BW24}.
  \end{enumerate}
\end{proof}
\begin{cor}\label{Hom-PI}
  If $E$ is a pure-injective object of $\Tcal$ then $[X,E]$ is pure-injective for any $X \in \Tcal$. In particular, if $S\in \Tcal ^c$ then $S\otimes E$ is also pure-injective.
\end{cor}
\begin{proof}
 Since $E$ is pure-injective, the natural map $E \to E^{\bc\bc}$ of \cref{BW-duality}(3) is a split monomorphism. It follows that $[X,E]$ is a retract in $[X,E^{\bc\bc}]$, and by \cref{BW-duality}(1), $[X, E^{\bc\bc}] \cong (X\otimes E^\bc)^\bc$, so the conclusion follows by \cref{BW-duality}(2).

 To prove the second part of the statement, let $S\in \Tcal ^c$. By our rigidity  hypothesis,  $S\otimes E\cong [S^*,E]$.  Hence, by the first part of the statement, $S\otimes E$ is  pure-injective.
\end{proof}
We say that two definable subcategories $\Ccal$ and $\Ccal^{\vee}$ are \newterm{dual definable} if there is a set $F$ of morphisms in $\Tcal^\cpt$ such that $\Ccal = F\Perp{0}$ and $\Ccal^{\vee} = (F^*)\Perp{0}$, where $F^* = \{f^* \colon f \in F\}$ is the dual set of morphisms in $\Tcal^\cpt$.
\begin{lemma}\label{dualdefinable}
  Let $\Ccal$ be a definable subcategory of $\Tcal$. Then $\Ccal^{\vee}$ is the unique definable subcategory of $\Tcal$ satisfying $X \in \Ccal$ if and only if $X^\bc \in \Ccal^{\vee}$.
\end{lemma}
\begin{proof}
  First we check that $\Ccal^\vee$ has the above property. Given a morphism $f$ in $\Tcal^\cpt$, we have using \cref{BW-duality}(4) that $\Hom_\Tcal(f^*,X^\bc)$ is zero if and only if $\Hom_\Tcal(f,X)^\cd$ is zero which is, by conservativeness of the functor $\Hom_\Zbb(-,\Qbb/\Zbb)\colon \Mod{\Zbb}^\op \to \Mod{\Zbb}$, equivalent to $\Hom_\Tcal(f,X)$ being zero. Thus, if $F$ is a set of morphisms in $\Tcal^\cpt$ such that $\Ccal = F\Perp{0}$ and $\Ccal^\vee = (F^*)\Perp{0}$, we see that $X \in \Ccal$ if and only if $X^\bc \in \Ccal^\vee$.

  Now let $\Dcal$ be another definable subcategory of $\Tcal$ and assume that $X \in \Ccal$ if and only if $X^\bc \in \Dcal$. This means that for any $X \in \Tcal$, $X^\bc \in \Dcal$ if and only if $X^\bc \in \Ccal^\vee$. By symmetry, this shows that $X \in \Dcal$ if and only if $X^{\bc\bc} \in \Ccal^\vee$, but the latter condition is equivalent to $X \in \Ccal^\vee$ by \cref{BW-duality}(3).
\end{proof}
\subsection{$\otimes$-co-$t$-structures}
\label{ss:cotstructres}
A \newterm{co-$t$-structure} on $\Tcal$ is a pair $(\Xcal,\Wcal)$ of subcategories of $\Tcal$ closed under direct summands such that the following axioms hold:
\begin{enumerate}
  \item[(coT1)] $\Wcal \subseteq \Xcal\Perp{0}$,
  \item[(coT2)] $\Wcal$ is closed under $\Sigma$,
  \item[(coT3)] $\Tcal = \Xcal \star \Wcal$.
\end{enumerate}
Following \cite{DS23b}, a co-$t$-structure $(\Xcal,\Wcal)$ in $\Tcal$ is a \newterm{$\otimes$-co-$t$-structure} if for any $S \in \Tcal^{\leq 0,\cpt}$ and $X \in \Xcal$ we have $[S,X] \in \Xcal$. As with $\otimes$-$t$-structures, it is defined relative to the fixed $t$-structure. 

Let us denote by $(\Pcal^{\geq 0},\WTcal{0})$ the co-$t$-structure generated by $(\Tcal^{\leq 0,\cpt})^*$, meaning that $\WTcal{0} = ((\Tcal^{\leq 0,\cpt})^*)\Perp{0}$. The pair $(\Pcal^{\geq 0},\WTcal{0})$ indeed forms a co-$t$-structure by \cite[Theorem 4.3, Corollary 4.6]{AI12}. Moreover, in the terminology of \cref{ss:bw-duality}, this means that $\WTcal{0} = ({\Tcal^{>0}})^\vee$ is the dual definable subcategory to the definable coaisle $\Tcal^{>0}$. We also let $\WTcal{n} \coloneqq \Sigma^{-n}\WTcal{0}$ for each $n \in \Zbb$. In fact, $(\Pcal^{\geq 0},\WTcal{0})$ is a $\otimes$-co-$t$-structure, as for any $S,U \in \Tcal^{\leq 0,\cpt}$ and $W \in \WTcal{0}$ we have $\Hom_\Tcal(U^*,[S^*,W]) \cong \Hom_\Tcal((S \otimes U)^*,W) = 0$ as $S \otimes U \in \Tcal^{\leq 0,\cpt}$. The following is the co-$t$-structure analogue of \cref{tensor-tstr}.
\begin{lemma}\label{tensor-cotstr}
 Let $(\Xcal,\Wcal)$ be a co-$t$-structure in $\Tcal$. The following are equivalent:
 \begin{enumerate}
  \item[(i)] $(\Xcal,\Wcal)$ is a $\otimes$-co-$t$-structure,
  \item[(ii)] for all $X \in \Xcal$ and $W \in \Wcal$ we have $[X,W] \in \WTcal{0}$.
\end{enumerate}
\begin{proof}
  $(i) \implies (ii)\colon$ For any $S \in\Tcal^{\leq 0,\cpt}$ we have $\Hom_\Tcal(S^*,[X,W]) \cong \Hom_\Tcal([S,X],W) = 0$, as the assumption yields $[S,X] \in \Xcal$ Therefore, $[X,W] \in \WTcal{0}$ as desired.

  $(ii) \implies (i)\colon$ Let $S \in \Tcal^{\leq 0,\cpt}$, $X \in \Xcal$, and $W \in \Wcal$, and let us show that $\Hom_\Tcal([S,X],W) = 0$. We have $\Hom_\Tcal([S,X],W) \cong \Hom_\Tcal(S^* \otimes X,W) \cong \Hom_\Tcal(S^*,[X,W])$, where the last group vanishes by $(ii)$.
\end{proof}
\end{lemma}
\subsection{Cohomological semi-Bousfield classes}
\label{ss:cohomologicalsb}
Let $\PI$ be the subcategory of $\Tcal$ consisting of all pure-injective objects. A \newterm{cohomological semi-Bousfield class}\footnote{This terminology is not compatible with the original paper of Hovey \cite{Hov95} which introduced cohomological Bousfield classes as full left orthogonals to \textit{arbitrary objects} of $\Tcal$. Our terminology is however in line with that of Krause and Stevenson, see \cite[Proposition 3.2.7]{KS19}.} is a subcategory of $\Tcal$ of the form $\CSBou{\Ecal} = \{X \in \Tcal \colon [X,E] \in \WTcal{0} ~\forall E \in \Ecal\}$ for some subcategory $\Ecal$ of $\PI$. The collection $\CSBL{\Tcal}$ of cohomological semi-Bousfield classes again admits a partial order by reverse inclusion.

Given a subcategory $\Ecal$ of $\PI$, set $\widehat{\Ecal} \coloneqq \{S \otimes E \colon S \in \Tcal^{\leq 0,\cpt}, E \in \Ecal\}$.
\begin{lemma}\label{csB-orth}
  For any $\Ecal \subseteq \PI$ we have $\CSBou{\Ecal} = \Perp{0}\widehat{\Ecal}$.
\end{lemma}
\begin{proof}
  For any $E \in \Ecal$ and $S \in \Tcal^{\leq 0, \cpt}$ we have the isomorphism $\Hom_\Tcal(X,S \otimes E) \cong \Hom_\Tcal(S^*,[X,E])$. If $X \in \CSBou{\Ecal}$, we have $[X,E] \in \WTcal{0}$, making the latter $\Hom$ group vanish, showing that $X \in \Perp{0}\widehat{\Ecal}$. Conversely, if $\Hom_\Tcal(S^*,[X,E]) = 0$ for all $S \in \Tcal^{\leq 0,\cpt}$, we have $[X,E] \in (\Tcal^{\leq 0,\cpt})^*\Perp{0} = \WTcal{0}$.
\end{proof}

\begin{lemma}\label{csBousfield-closure}
  Let $~\Xcal$ be a subcategory of $~\Tcal$. Then $\CSBou{\Xcal}$ is closed under coproducts, $\Sigma^{-1}$, extensions, pure subobjects, and pure quotients. If $~\Tcal$ has a model, then $\CSBou{\Xcal}$ is closed under filtered homotopy colimits.
\end{lemma}
\begin{proof}
  The proof follows identically to \cref{SBousfield-closure}, and because pure-injectives are mapped to injectives by the restricted Yoneda embedding, see \cref{ss:definable}. 
\end{proof}

\begin{lemma}\label{cBousfield-set}
    The following hold. 
  \begin{enumerate}
    \item[(i)] The cohomological semi-Bousfield classes of $\Tcal$ form a set.
    \item[(ii)] Let $\Ecal$ be a subcategory of $\PI$. Then there is an object $E \in \PI$ such that $\CSBou{E} = \CSBou{\Ecal}$.    
  \end{enumerate}
\end{lemma}
\begin{proof}
  Let $\Ccal = \CSBou{\Ecal}$. By \cref{csB-orth} and \cref{Hom-PI}, we may assume that $\Ccal = \Perp{0} \Ecal$. Then there is a localising subcategory $\Lcal$ of the module category $\Mod{\Tcal^\cpt}$ such that 
  $$\Ccal = \{X \in \Tcal \colon \yo(X) \in \Lcal\},$$ as $\Ccal$ is closed under pure subobjects, pure quotients, and coproducts. Recalling the classical results in \cite{Gab62}, in any given Grothendieck category, there is only a set of localising subcategories, we verified $(i)$.
  
  To prove $(ii)$, it suffices to observe that given a localising subcategory $\Lcal$ of a Grothendieck category $\Gcal$, there is an injective object $W \in \Gcal$ such that $\Lcal = \{M \in \Gcal \colon \Hom_\Gcal(M,W) = 0\}$, then we just let $E \in \PI$ be such that $\yo (E) = W$. To establish the observation, let $\alpha$ be a cardinal such that any object $M \in \Gcal$ is a direct union of its $\alpha$-generated subobjects. Let $W$ be the product of injective envelopes of a representative set of isomorphism classes of those objects from $\Lcal^\perp$ which are $\alpha$-generated. It is easily observed that any object of $\Lcal^\perp$ then embeds into a product of copies of $W$, establishing the claim.
\end{proof}

\begin{lemma}\label{tensorstr-t-costr}
  The assignment $\SBou{X} \mapsto \CSBou{X^\bc}$ yields a poset embedding $\SBL{\Tcal} \hookrightarrow \CSBL{\Tcal}$.
  
  Assume $\Tcal$ is algebraic. Then any cohomological semi-Bousfield class $\Xcal$ fits into a $\otimes$-co-$t$-structure $(\Xcal,\Wcal)$ as a left-hand constituent.
\end{lemma}
\begin{proof}
  To prove the first assertion, we claim that $\SBou{X} = \CSBou{X^\bc}$. Indeed, we have $[Y,X^\bc] \cong (X \otimes Y)^\bc$ by \cref{BW-duality}(1) and thus $[Y,X^\bc] \in \WTcal{0}$ if and only if $X \otimes Y \in \Tcal^{>0}$ by \cref{dualdefinable}, as desired. Note that since $(\coprod_{i \in I}X_i)^\bc \cong \prod_{i \in I}X^\bc$ by \cite[Corollary 3.13]{CH09}, this embedding preserves infinite joins.
  
  Now assume $\Tcal$ is algebraic and let $\Xcal = \CSBou{\Ecal}$ be the cohomological semi-Bousfield class of a subcategory $\Ecal$ of $\PI$. We claim that $\Xcal = \Perp{0}\Vcal$, where $\Vcal = \{[S^*,E] \colon E \in \Ecal, S \in \Tcal^{\leq 0,\cpt}\}$. Indeed, the isomorphism $\Hom_\Tcal(X,[S^*,E]) \cong \Hom_\Tcal(S^*,[X,E])$ shows that $X \in \Perp{0}\Vcal$ if and only if $[X,E] \in \WTcal{0}$ for all $E \in \Ecal$, and thus if and only if $X \in \CSBou{\Ecal}$. Since objects of the form $[S^*,E]$ are also pure-injective, \cite[Corollary 5.4]{LV20} yields the existence of a co-$t$-structure $(\Xcal,\Vcal)$, here we need that $\Tcal$ is algebraic. The fact that the co-$t$-structure $(\Xcal,\Vcal)$ is a $\otimes$-co-$t$-structure follows directly from the computation above and \cref{tensor-cotstr}.
\end{proof}
\begin{rmk}
    The relationship between compactly generated $\otimes$-$t$-structures and compactly generated $\otimes$-co-$t$-structures, was shown in the presence of a model in \cite[Lemma 3.5 and Theorem 3.10]{DS23b}, following the non-tensor compatible version shown in \cite[Theorem 4.10]{SP16}.
\end{rmk}
\begin{rmks}
  $(1)\colon$ In general, not every cohomological Bousfield class is a Bousfield class, as first noted by Stevenson \cite{Ste14}. 
     
  \noindent$(2)\colon$ Also, not every aisle in a $\otimes$-co-$t$-structure is a cohomological semi-Bousfield class. In fact, in the category of abelian groups, there are class-many complete cotorsion pairs \cite[Example 5.16]{GT12}. It follows that there are class-many ($\otimes$-)co-$t$-structures in the unbounded derived category of the integers, as one can construct an injective assignment from the complete cotorsion pairs to co-$t$-structures. 
\end{rmks}
We sum up our findings so far in the following diagram. A \newterm{localising $\otimes$-ideal} is a thick $\otimes$-ideal $\Lcal$ of $\Tcal$ closed under $(X \otimes -)$ for any $X \in \Tcal$ (note that this forces $\Lcal$ to be closed under all coproducts in $\Tcal$). Recall from \cite[Proposition 3.2.7]{KS19} that a cohomological localising $\otimes$-ideal is a localising $\otimes$-ideal which is the full left orthogonal to a subcategory consisting of pure-injective objects of $\Tcal$.
$$
\begin{tikzcd}[column sep = 0.1em, row sep = 0.1em]
 & & \begin{Bmatrix} 
  \text{ Bousfield} \\ \text{classes} \end{Bmatrix} & \subseteq & \begin{Bmatrix} 
    \text{ Cohomological localising} \\ \text{$\otimes$-ideals} \end{Bmatrix} \\
  & & \rotatebox[origin=c]{270}{$\subseteq$} & & \rotatebox[origin=c]{270}{$\subseteq$}\\
\begin{Bmatrix} 
  \text{ Coaisles of compactly} \\ \text{generated $\otimes$-$t$-structures} \end{Bmatrix}& \subseteq & \begin{Bmatrix} \text{Semi-Bousfield} \\ \text{classes} \end{Bmatrix} & \subseteq & \begin{Bmatrix} \text{Cohomological semi-Bousfield} \\ \text{classes} \end{Bmatrix}\\
\end{tikzcd}
$$

\subsection{Standard pairs and dimensions}\label{ss:dimensions}
\cref{cgtstrtotensor} shows that the semi-Bousfield pair $(\Tcal^{>0},\Fcal^{\geq 0})$ has as right consistuent the class
$$\Fcal^{\geq 0} = \{X \in \Tcal \colon X \otimes \Tcal^{>0} \subseteq \Tcal^{>0}\}.$$
We shall also denote $\Fcal^{\geq -n} = \Sigma^{n}\Fcal^{\geq 0}$ and think of its objects as being of ``flat dimension'' at most $n$. Dually, we let
$$\Ical^{\leq 0} = \{X \in \Tcal \colon [\Tcal^{>0},X] \subseteq \WTcal{0}\}.$$
We shall also denote $\Ical^{\leq n} = \Sigma^{-n}\Ical^{\leq 0}$ and think of its objects as being of ``injective dimension'' at most $n$. Also let $\Fcal^+ = \bigcup_{n \in \Zbb}\Fcal_n $ denote the subcategory of all objects ``of finite flat dimension'' and $\Ical^- = \bigcup_{n \in \Zbb} \Ical_n$ denote the subcategory of all objects of ``finite injective dimension''. A justification of this nomenclature in the algebro-geometric setting will arrive later in \cref{scheme-dimensions}.

\begin{lemma}\label{FIdual}
  Let $X \in \Tcal$. Then:
  \begin{enumerate}
    \item[(i)] $X \in \Fcal^{\geq n}$ if and only if $X^\bc \in \Ical^{\leq -n}$ if and 
    only if $X^\bc \in \PI \cap \Ical^{\leq -n}$.
    \item[(ii)] $X \in \Fcal^{+}$ if and only if $X^\bc \in \Ical^{-}$ if and 
    only if $X^\bc \in \PI \cap \Ical^{-}$.
    \item[(iii)] if $X \in \Fcal^+$ and $Y \in \Ical^-$ then $[X,Y] \in \Ical^-$.
    \item[(iv)] $X \in \Fcal^+$ if and only if $(X \otimes -)$ preserves $\Tcal^+$.
  \end{enumerate}
\end{lemma}
\begin{proof}
  We have the natural isomorphism $[X \otimes -]^\bc \cong [-,X^\bc]$. Therefore, $X^\bc \in \Ical^{\leq 0}$ if and only if $[X \otimes -]^\bc$ sends $\Tcal^{>0}$ to $\WTcal{0}$ if and only $X \otimes -$ sends $\Tcal^{>0}$ to $\Tcal^{>0}$ if and only if $X \in \Fcal^{\geq 0}$. The rest follows easily by shifting. To prove $(iii)$, let $Z \in \Tcal^{>0}$ and note that $[Z,[X,Y]] \cong [Z \otimes X,Y]$. Since $X \in \Fcal^+$, there is $n \in \Zbb$ such that $Z \otimes X \in \Tcal^{>n}$. Therefore, $Y \in \Ical^-$ yields some $m \in \Zbb$ such that $[Z,[X,Y]] \in \WTcal{m}$, establishing that $[X,Y] \in \Ical^-$. We leave $(iv)$ to the reader.
\end{proof}

\begin{lemma}\label{F0preserve}
  Let $\Ccal$ be a cohomological semi-Bousfield class and $F \in \Fcal^{\geq 0}$. Then the functor $(F \otimes -)$ preserves $\Ccal$.
\end{lemma}
\begin{proof}
  Let $\Ecal \subseteq \PI$ be such that $\Ccal = \CSBou{\Ecal}$ and let $C \in \Ccal$. Then we have for any $E \in \Ecal$ that $[F \otimes C,E] \cong [F,[C,E]]$. By the assumption, $[C,E] \in \WTcal{0}$. By \cref{Hom-PI}, $E \in \PI$ implies $[C,E] \in \PI$. It thus remains to show that $[F,-]$ preserves $\WTcal{0} \cap \PI$. Any object $W \in \WTcal{0} \cap \PI$ is a direct summand of an object of the form $X^\bc$ for $X \in \Tcal^{> 0}$. Additionally, $[F,X^\bc] \cong (F \otimes X)^\bc$, and as $F \otimes -$ preserves $\Tcal^{> 0}$, $(F \otimes X)^\bc \in \WTcal{0}$, so we are done.
\end{proof}

We also remark on the following interaction of finite flat dimension with compact objects.

\begin{lemma}\label{compactsflat}
  For any compact object $S \in \Tcal^{\leq 0, \cpt}$ the dual $S^*$ belongs to $\Fcal^{\geq 0}$.
  
  Any compact object $S \in \Tcal^\cpt$ belongs to $\Fcal^{+}$ and to $\Tcal^-$.
\end{lemma}
\begin{proof}
  For the first claim, let $S \in \Tcal^{\leq 0, \cpt}$ and let us check that $S^* \otimes - $ preserves $\Tcal^{>0}$. For any $X \in \Tcal^{>0}$, we have $S^* \otimes M \cong [S,M]$, and $[S,M] \in \Tcal^{>0}$ by \cref{tensor-tstr}.
  
  Now let us prove the second claim. The assumption $\bigcap_{n \in \Zbb}\Tcal^{>n} = 0$ is equivalent to $\thick{\Tcal^{\leq 0,\cpt}} = \Tcal^\cpt$ as $(\Tcal^{\leq 0}, \Tcal^{>0})$ is compactly generated, see \cite[Example 1, Corollary 9]{MurE}. It follows that there is $n \in \Zbb$ such that $\Sigma^n S \in \Tcal^{\leq 0, \cpt}$. This immediately yields $S \in \Tcal^-$. By the same argument, there is $k \in \Zbb$ such that $\Sigma^k S^* \in \Tcal^{\leq 0,\cpt}$, and thus $S^{**} \cong S \in \Fcal^{+}$ by the first claim.
\end{proof}

\begin{ex}
Consider the setting of the stable homotopy category $\SH$ of spectra. The choice of the base $t$-structure $(\Tcal^{\leq 0},\Tcal^{>0})$ is the Postnikov $t$-structure $(\SH^{\leq 0},\SH^{>0})$ whose induced $i$-th cohomology functor is the $i$-th homotopy group functor $\pi^i$ (to be consistent with our notation, we are committing the heresy of using cohomological indexing here), that is, $\SH^{\leq 0}$ consists of connective spectra and $\SH^{\geq 0}$ consists of coconnective spectra. Since the unit object, the sphere spectrum $S$, generates $\SH^{\leq 0}$ as an aisle, every (co-)$t$-structure in $\Scal\Hcal$ is automatically a $\otimes$-(co-)$t$-structure. Since $S$ is a silting object of $\Scal\Hcal$, standard arguments show that $\WTcal{0} = \SH^{<0}$ in this case. Then $\Pcal^{\geq 0}$ consists of spectra $X$ such that $[X,-]$ preserves connective spectra, $\Fcal^{\geq 0}$ consists of spectra $X$ such that $X \wedge -$ preserves coconnective spectra, and $\Ical^{\leq 0}$ consists of spectra $X$ such that $[-,X]$ sends coconnective spectra to connective spectra.
\end{ex}

\section{Generalities on $\D_{\qc}(X)$}\label{s:schemes}
In this section, let $X$ be a quasi-compact and quasi-separated scheme, $\Ocal_X$ denote the structure sheaf on $X$, $\Mod{X}$ the category of sheaves of $\Ocal_X$-modules, and $\Qcoh{X}$ its full subcategory of quasi-coherent sheaves. We let $\D(\Mod{X})$ denote the unbounded derived category of the Grothendieck category $\Mod{X}$. The setting of the big tt-category $\Tcal$ of the previous section will be specialised to the specific choice 
$$\D = \D_{\qc}(X),$$ 
the full subcategory of $\D(\Mod{X})$ of cochain complexes of $\Ocal_X$-modules with cohomology in $\Qcoh{X}$. The inclusion of $\Qcoh{X}$ into $\Mod{X}$ has a right adjoint 
$$Q\colon \Mod{X} \to \Qcoh{X}$$ 
called the \newterm{coherator}. Its right derived functor $$\RQ\colon \D(\Mod{X}) \to \D(\Qcoh{X})$$ 
restricts as 
$$\RQ\colon \D_{\qc}(X) \to \D(\Qcoh{X}),$$ 
and this latter functor is an equivalence for any scheme which is either semi-separated or otherwise Noetherian, see \cite[08DB, 08D7, 09T4]{Stacks}. Thus, in either of the two cases, $\D$ identifies with the derived category $\D(\Qcoh{X})$ of the Grothendieck category $\Qcoh{X}$. We now recall the big ttt-category structure on $\Tcal$ and some other useful notions. The category $\D$ is compactly generated and the subcategory $\D^\cpt$ of compact objects consists, up to isomorphism, of perfect complexes, see \cite{BVdB03}. In several places of this section, we will restrict to the case of $X$ being either semi-separated or Noetherian, and for the main results of the next sections, we will restrict to $X$ Noetherian (but not necessarily semi-separated).

\subsection{Tensor product and hoflats} The derived category $\D(\Mod{X})$ is endowed with the left derived tensor product $- \otimes_X^\mathbf{L} -$, computed using hoflat resolutions and the sheaf tensor product $- \otimes_X -$, which is the standard symmetric monoidal product on $\Mod{X}$. A cochain complex $F$ over $\Mod{X}$ is \newterm{hoflat} if $F \otimes_X N$ is acyclic for any acyclic complex $N$. A sheaf $F \in \Mod{X}$ is \newterm{flat} if and only if it is hoflat as a complex concentrated in degree zero, equivalently, if the induced tensor functor $F \otimes_X -$ is exact on $\Mod{X}$. Any bounded above complex of flat sheaves is hoflat and any complex is quasi-isomorphic to a hoflat complex with flat components \cite[Proposition 46, Proposition 48]{MurC}, \cite[05NI]{Stacks}. Then for any $M,N \in \D(\Mod{X})$, the derived tensor product $M \otimes_X^\mathbf{L} N$ is isomorphic to $F_M \otimes_X N$, where $F_M$ is any hoflat quasi-isomorphic replacement of $M$. This endows $\D(\Mod{X})$ with a symmetric monoidal product $- \otimes_X^\mathbf{L}-$, see \cite[Theorem 1.1]{CH11} with tensor unit $\1 = \Ocal_X$ being the structure sheaf. Furthermore, the derived tensor product restricts as a functor 
$$- \otimes_X^\mathbf{L} -\colon \D_{\qc}(X) \times \D_{\qc}(X) \to \D_{\qc}(X),$$ 
see \cite[2.5.8]{Lip60}, endowing $\D$ with a symmetric monoidal product. Here, note that $- \otimes_X^\mathbf{L} -$ is still computed using hoflat resolutions which are complexes in $\D_{\qc}(X)$ rather then complexes with components in $\Qcoh{X}$. Indeed, $\Qcoh{X}$ fails to contain enough flat sheaves unless $X$ is semi-separated, see \cite[Main Theorem]{SS21}.

\subsection{Internal Hom and hoinjectives} That the symmetric monoidal structure is closed follows from a suitable Brown representability property of $\D$, let us spell out how the internal Hom is computed explicitly. The symmetric monoidal category $\Mod{X}$ is endowed with an internal Hom functor, the sheaf Hom functor 
$$\SHom_X(-,-)\colon \Mod{X}^\op \times \Mod{X} \to \Mod{X}.$$ 
A complex $I$ of sheaves of $\Ocal_X$-modules is \newterm{hoinjective} if $\Hom_{\K(\Mod{X})}(A,I) = 0$ for any acyclic complex $A$. As a Grothendieck category, $\Mod{X}$ has enough injectives. It follows from standard results \cite[Proposition 47, Theorem 113]{MurB}, \cite[070L]{Stacks} that any bounded below complex of injective sheaves is hoinjective and that any complex is quasi-isomorphic to a hoinjective complex with injective components. The right derived functor 
$$\RSHom_X(-,-)\colon \D(\Mod{X})^\op \times \D(\Mod{X}) \to \D(\Mod{X})$$ 
is computed using hoinjectives in the sense that for any $M,N \in \D(\Mod{X})$, $\RSHom_X(M,N)$ is isomorphic to $\SHom_X(M,I_N)$ for any quasi-isomorphic hoinjective replacement $I_N$ of $N$, \cite[Lemma 23]{MurC}. The derived sheaf Hom functor $\RSHom_X(-,-)$ is the internal Hom of the symmetric monoidal category $\D(\Mod{X})$, \cite[Theorem 1.1]{CH11}. However, this internal Hom does not restrict to $\D_{\qc}(X)$. Instead, the internal Hom 
$$[-,-]\colon \D_{\qc}(X)^\op \times \D_{\qc}(X) \to \D_{\qc}(X)$$ 
is computed as the composition $\RQ \circ \RSHom_X(-,-)$ of the derived sheaf Hom and the right derived coherator functor, \cite[Theorem 1.2]{CH11}. In the case $X$ is semi-separated, the functor $\RQ \circ \RSHom_X(-,-)$ can be computed using hoinjectives in the sense that for any $M,N \in \D$, $\RQ \circ \RSHom_X(M,N)$ is isomorphic to $\SHomqc_X(M,N) = Q \circ \SHom_X(M,I_N)$ for any quasi-isomorphic hoinjective replacement $I_N$ of $N$, \cite[Lemma 4.10]{AJPV18}. Here, $\SHomqc_X(-,-)$ is the internal Hom of the symmetric monoidal category $\Qcoh{X}$ endowed with the restriction of $- \otimes_X -$, see \cite[01CE]{Stacks}. Note that the injective objects of $\Qcoh{X}$ may fail to be injective with respect to the internal Hom functor $\SHomqc_X(-,-)$ unless $X$ is semi-separated, \cite[Main Theorem]{SS21}. If $X$ is Noetherian, then we can compute hoinjective resolutions $\D(\Qcoh{X})$ in the following sense: Each complex $M \in \D$ is isomorphic to a hoinjective complex with components being injective quasi-coherent sheaves, by \cite[09T4]{Stacks} and the fact that the inclusion $\Qcoh{X} \hookrightarrow \Mod{X}$ then preserves injectives, \cite[II.7.17]{Har66}, see also \cite[070L]{Stacks}.

\begin{lemma}\label{hoinjectives}
  Let $X$ be either semi-separated or Noetherian. Then the following are equivalent for any $M \in \D$:
  \begin{enumerate}
    \item[(i)] $M \in \D^{>0}$,
    \item[(ii)] $\Hom_\D(\Sigma^i F,M) = 0$ for any $F \in \Qcoh{X}$ and $i \geq 0$. 
  \end{enumerate}
  Let $X$ be Noetherian. Then the following are equivalent for any $M \in \D$:
  \begin{enumerate}
    \item $M$ is isomorphic to a hoinjective complex of injective quasi-coherent sheaves concentrated in nonpositive degrees,
    \item $\Hom_\D(\Sigma^i F,M) = 0$ for any $F \in \Qcoh{X}$ and $i < 0$. 
  \end{enumerate}
\end{lemma}
\begin{proof}
  $(i) \iff (ii)\colon$ As $\Sigma^i F \in \D^{\leq 0}$ for any $i \geq 0$, the implication $(i) \implies (ii)$ is clear and does not require any assumptions, let us prove $(ii) \implies (i)$. By the hypothesis on $X$, we can assume that $M$ has quasi-coherent components. Let $F \coloneqq Z^i(M)$ be the $i$-th sheaf of cocycles of $M$. Then the induced morphism $\Sigma^i F \to M$ is nonzero whenever $H^i(M) \neq 0$.

  $(1) \iff (2)\colon$ This follows from a standard argument in derived categories of Grothendieck categories applied to $\D(\Qcoh{X}) \cong \D$.
\end{proof}

\vspace{1em}


\subsection{Restriction, stalks, and the derived costalk functor}
Let $i_U\colon U \hookrightarrow X$ be an inclusion of an open set. Then we have the exact restriction functor 
$$(-)_{\restriction U}\colon \Qcoh{X} \to \Qcoh{U}$$ 
by \cite[Lemma 22]{MurD}. Assume further that $U$ is quasi-compact, then $(-)_{\restriction U}$ admits a right adjoint $(i_U)_*\colon \Qcoh{U} \to \Qcoh{X}$, the direct image. Note that as $(-)_{\restriction U}$ is exact, the direct image $(i_U)_*$ preserves injective objects. Furthermore, the direct image functor $(i_U)_*$ is left exact and its right derived functor 
$$\mathbf{R}(i_U)_*\colon \D(U) \coloneqq \D(\Qcoh{U}) \to \D$$ 
is the right adjoint to the trivially derived restriction functor $(-)_{\restriction U}\colon \D \to \D(U)$. If $i_U\colon U \to X$ is affine, e.g. if $U$ is affine and $X$ is semi-separated, then $(i_U)_*$ is exact \cite[0G9R]{Stacks}.

For any point $x \in X$ let $\Ocal_{X,x} = \varinjlim_{x \in U \subset X}\Ocal_X(U)$ denote the local commutative Noetherian ring called the \newterm{stalk} of $\Ocal_X$ at $x$. Then we have an inclusion of schemes map $i_x\colon \Spec{\Ocal_{X,x}} \hookrightarrow X$ obtained as the pre-composition of the inclusion $i_U\colon U \hookrightarrow X$ for any open affine subset $U$ of $X$ such that $x \in U$ with the affine scheme inclusion $\Spec{\Ocal_{X,x}} \hookrightarrow U \cong \Spec{\Ocal_X(U)}$ induced by the localisation ring map $\Ocal_X(U) \to \Ocal_{X,x}$. The inverse image functor 
$$(i_x)^*\colon \Qcoh{X} \to \Qcoh{\Spec{\Ocal_{X,x}}} \cong \Mod{\Ocal_{X,x}}$$ 
induced by $i_x$ is just the exact stalk functor assigning to a quasi-coherent sheaf $M$ the $\Ocal_{X,x}$-module $M_x \coloneqq \varinjlim_{x \in U \subset X}M(U)$. The direct image functor $(i_x)_*$ thus preserves injective quasi-coherent sheaves and hoinjective complexes. Furthermore, $(i_x)_*$ is left exact, and it is exact in the case that $X$ is semi-separated, as that assumption renders $i_x$ an affine scheme map. On the level of derived categories, we have an adjoint triple

$$
\begin{tikzcd}[column sep=large]
  \D(\Ocal_{X,x}) \arrow{r}{\mathbf{R}(i_x)_*} & \D \arrow[bend right=60]{l}[swap]{(-)_x} \arrow[bend left=60]{l}{(-)^x}
\end{tikzcd}
$$
where $(-)_x$ is the trivially derived exact stalk functor, $\mathbf{R}(i_x)_*$ is its right adjoint, the right derived direct image. The existence of the right adjoint $(-)^x$ to the derived direct image follows abstractly by Brown representability, see \cite{Nee96}. The upshot is that, since $(-)_x$ preserves compact objects, $\mathbf{R}(i_x)_*$ preserves coproducts, which by Brown representability yields the existence of its right adjoint, which we denote as $(-)^x$, and interpret it as the \newterm{derived costalk functor}.

The \newterm{(sheaf) support} of a quasi-coherent sheaf $M \in \Qcoh{X}$ is defined as 
$$\Supp(M) = \{x \in X \colon M_x \neq 0\}.$$
If $M \in \D$, we extend this by $\Supp(M) = \{x \in X \colon M_x \neq 0\}$, which boils down to the sheaf version via cohomology using the exactness of the stalk functor: 
$$\Supp(M) = \Supp(\coprod_{i \in \Zbb}H^i(M)).$$
\subsection{Standard $t$-structure}
The role of the base $t$-structure $(\Tcal^{\leq 0},\Tcal^{>0})$ will be played by the \newterm{standard $t$-structure} $(\D^{\leq 0},\D^{>0})$ corresponding to the cochain cohomology functor $H^0$, with the truncation triangles given by the soft truncations of cochain complexes $(-)^{\leq 0}$ and $(-)^{>0}$. Clearly, the aisle $\D^{\leq 0}$ contains the unit object $\Ocal_X$ and is closed under the left derived tensor product $- \otimes_{X}^\mathbf{L} -$. Finally, the standard $t$-structure $(\D^{\leq 0},\D^{>0})$ is compactly generated, at least if $X$ is Noetherian, this follows from the classification result of \cite{DS23}, but we also provide a direct proof below. As per usual, $\coh{X} \subseteq \Qcoh{X}$ denotes the subcategory of coherent sheaves. Recall that the category $\Qcoh{X}$ is locally finitely presentable, and if $X$ is Noetherian then the finitely presentable objects are precisely the coherent sheaves \cite[0GN6,01XZ]{Stacks}.

\begin{lemma}\label{perfect-approx}
  Let $X$ be a Noetherian scheme and $E \in \coh{X}$. Then there is $S \in \D^{\leq 0,\cpt}$ such that $\Supp(S) = \Supp(E)$ and $H^0(S) \cong E$.
\end{lemma}
\begin{proof}
  Since $E$ is coherent, $\Supp(E)$ is a closed subset of $X$ with quasi-compact complement. Then we can apply the approximation by perfect complexes theorem of Lipman and Neeman \cite[Theorem 4.1]{LN07}, or more specifically the incarnation \cite[08ES, 08EN]{Stacks} to $E$ with the choice of parameters $T = \Supp(E)$ and $m=-1$. This yields the desired compact object $S$ such that $\Supp(S) \subseteq \Supp(E)$ and such that the cohomology of $S$ and $E$ coincides in nonnegative degrees, which yields both $S \in \D^{\leq 0,\cpt}$ and $H^0(S) \cong E$.
\end{proof}

\begin{prop}
  Let $X$ be a Noetherian scheme. Then the standard $t$-structure $(\D^{\leq 0},\D^{>0})$ on $\D_{\qc}(X)$ is compactly generated.
\end{prop}
\begin{proof}
  It is clear that $\D^{>0} \subseteq (\D^{\leq 0, \cpt})\Perp{0}$, it remains to show the converse inclusion. Let $M \in  (\D^{\leq 0, \cpt})\Perp{0}$ and assume towards contradiction that $H^i(M) \neq 0$ for some $i \leq 0$. By the assumption, we can assume that $M$ is represented by a complex of quasi-coherent sheaves. In particular, the $i$-th cocycle sheaf $Z \coloneqq Z^i(M)$ belongs to $\Qcoh{X}$. Then there is $E \in \coh{X}$ and a morphism $E \to Z$ such that the composition $E \to Z \to H^i(M)$ is nonzero. Let $S \in \D^{\leq 0,\cpt}$ be the compact object such that $H^0(S) \cong E$ provided by \cref{perfect-approx}. We claim that there is a nonzero morphism $\Sigma^{-i} S \to M$, a contradiction with $M \in (\D^{\leq 0, \cpt})\Perp{0}$. Consider the morphism obtained by taking the composition 
  $$\Sigma^{-i}(S) \to \Sigma^{-i}(S)^{\geq i} \cong \Sigma^{-i}(E) \to \Sigma^{-i}(Z) \to M,$$ 
  where $\Sigma^{-i} Z \to M$ is just the inclusion of the cocycle sheaf. Then this is a nonzero map, as its $i$-th cohomology map $H^i(S) \cong E \to Z \to H^i(M)$ is nonzero.
\end{proof}

\begin{center}
\textit{It follows that, for a Noetherian scheme $X$, $\D$ is a big ttt-category in the sense of \cref{s:semibousfield}.  }
\end{center}

\subsection{Brown-Comenetz Duality}
It will be useful to compute the Brown-Comenetz duality $(-)^\bc$ in $\D$ explicitly. By definition, $\Ocal_X^\bc$ is determined by the natural isomorphism of functors $\D^\op \to \Mod{\Zbb}$
$$\Hom_\D(-,\Ocal_X^\bc) \cong \Gamma(X,-)^\cd,$$
because $\Hom_\D(\Ocal_X,-)$ coincides with the global sections functor 
$$\Gamma(X,-)\coloneqq(-)(X)\colon \Qcoh{X} \to \Mod{\Ocal_X(X)}.$$ 
For a general object $M \in \D$, $M^\bc$ is defined by the natural isomorphism
$$\Hom_\D(-,M^\bc) \cong \Gamma(X,- \otimes_X^\mathbf{L} M)^\cd,$$
and we also have the isomorphism $M^\bc \cong [M,\Ocal_X^\bc]$.

We first show that the Brown-Comenetz duality is computed via the character duality on $\Qcoh{U} \cong \Mod{\Ocal_X(U)}$ when restricted to an open affine subset $U$.

\begin{lemma}\label{dual-affine}
    The following hold. 
    \begin{itemize}
        \item[(i)] If $X$ is affine then the Brown-Comenetz duality is equivalent to the character duality.
        \item[(ii)]For any $U \subseteq X$ quasi-compact open we have an isomorphism natural in $M \in \D(U)$
  $$\mathbf{R}(i_U)_*(M^\bc) \cong \mathbf{R}(i_U)_*(M)^\bc,$$
  where $M^\bc$ denotes the Brown-Comenetz dual of $M$ computed inside $\D(U)$.
    \end{itemize}
\end{lemma}
\begin{proof}
  $(i)\colon$ This is \cref{BC-ring}. More explicitly, assume $X = \Spec{R}$. We have for any $M \in \D$ the natural isomorphisms $\Hom_\D(M,R^\cd) = \Hom_\D(M,\Hom_{\Zbb}(R,\Qbb/\Zbb)) \cong \Hom_\Zbb(M,\Qbb/\Zbb)) \cong \Gamma(X,M)^\cd$.

  $(ii)\colon$ It is enough to show that $\mathbf{R}(i_U)_*(M^\cd)$ satisfies the same natural isomorphism which defines the Brown-Comenetz dual $\mathbf{R}(i_U)_*(M)^\bc$. We have a chain of isomorphisms for any $N \in \D$:
  $$\Hom_\D(N,\mathbf{R}(i_U)_*(M^\cd)) \cong \Hom_{\D(U)}(N_{\restriction U},M^\bc) \cong \Gamma(U,N_{\restriction U} \otimes_U^\mathbf{L} M)^\cd \cong$$
  $$\cong \Gamma(X,\mathbf{R}(i_U)_*(N_{\restriction U} \otimes_U^\mathbf{L} M))^\cd \cong  \Gamma(X,N \otimes_X^\mathbf{L} \mathbf{R}(i_U)_*(M))^\cd \cong \Hom_\D(N,\mathbf{R}(i_U)_*(M)^\bc). $$  
  The first isomorphism comes from adjunction, the second from the definition of Brown-Comenetz duality in $\D(U)$, the third one is the definition of direct images, fourth one is \cite[Theorem 3.7]{CH11}, and the last one is the by the definition of Brown-Comenetz duality in $\D$.
\end{proof}

We first compute an explicit complex representing $\Ocal_X^\bc$ in the case when $X$ is semi-separated. Denote $[n] = \{1,2,\ldots,n\}$ and let $X = \bigcup_{i \in [n]} U_i$ be a cover of $X$ by open affine sets. For any $\lambda \subseteq [n]$ set $U_\lambda = \bigcap_{i \in \lambda} U_i$ and let us simplify the notation by letting $\Ocal_\lambda$ denote the structure sheaf of the scheme $U_\lambda$.

For $\lambda \subseteq {\lambda'}$, consider the restriction morphism $f_{\lambda,{\lambda'}}\colon \Ocal_{\lambda} \to \Ocal_{{{\lambda'}}}$ viewed as morphism in $\Qcoh{\Ocal_{\lambda}} \cong \Mod{\Ocal_{X}(U_\lambda)}$. Applying the character duality in $\Mod{\Ocal_{U_X}(U_\lambda)}$, we obtain the dual morphism $f_{\lambda,{\lambda'}}^\cd\colon \Ocal_{\lambda'}^\cd \to \Ocal_{\lambda}^\cd$ between injective modules. Define 
$$g_{\lambda,{\lambda'}} \coloneqq \mathbf{R}(i_U)_*(f^\cd) = (i_U)_*(f^\cd)\colon (i_U)_*(\Ocal_{\lambda'}^\cd) \to (i_U)_*(\Ocal_{\lambda}^\cd)$$ 
to be its direct image in $\D$, a morphism between injective objects of $\Qcoh{X}$. Let $\Lambda_k$ denote the set of all subsets of $\{1,2,\ldots,n\}$ of cardinality precisely $k$ and use the morphisms $g_{\lambda,{\lambda'}}$ for $\lambda \in \Lambda_k$ and ${\lambda'} \in \Lambda_{k+1}$ such that $\lambda \subseteq {\lambda'}$ to assemble the following induced cochain complex

$$(i_{U_{[n]}})_*\Ocal_{[n]}^\cd \to \prod_{\lambda \in \Lambda_{n-1}}(i_{U_{\lambda}})_*\Ocal_{\lambda}^\cd \to \prod_{\lambda \in \Lambda_{n-2}}(i_{U_{\lambda}})_*\Ocal_{\lambda}^\cd \to \cdots \to \prod_{\lambda \in \Lambda_{1}}(i_{U_{\lambda}})_*\Ocal_{\lambda}^\cd$$
of quasi-coherent sheaves of $\Ocal_X$-modules concentrated in degrees $-(n-1)$ to $0$ as an object of $\D$. Denote this complex by $\check{C}^\bc$. We claim that $\check{C}^\bc$ is isomorphic in $\D$ to $\Ocal_X^\bc$ by showing that it satisfies the same defining natural isomorphism as $\Ocal_X^\bc$. Since $\check{C}^\bc$ is a bounded complex of injectives it is hoinjective, and thus the homomorphism group $\Hom_\Dcal(M,\check{C}^\bc)$ for any $M \in \D$ is computed as the zero cohomology of following chain complex of abelian groups, using \cref{dual-affine}:
$$\Hom_\Dcal(M,((i_{U_{[n]}})_*\Ocal_{[n]})^\bc) \to \prod_{\lambda \in \Lambda_{n-1}}\Hom_\Dcal(M,((i_{U_{\lambda}})_*\Ocal_{\lambda})^\bc)  \to \cdots \to \prod_{\lambda \in \Lambda_{1}}\Hom_\Dcal(M,((i_{U_{\lambda}})_*\Ocal_{\lambda})^\bc),$$
which is by the definition of the Brown-Comenetz duality further isomorphic to
$$\Gamma(X,M \otimes_X^\mathbf{L} (i_{U_{[n]}})_*\Ocal_{[n]})^\cd \to (\coprod_{\lambda \in \Lambda_{n-1}}\Gamma(X,M \otimes_X^\mathbf{L} (i_{U_{\lambda}})_*\Ocal_{\lambda}))^\cd \to \cdots \to (\coprod_{\lambda \in \Lambda_{1}}\Gamma(X,M \otimes_X^\mathbf{L} (i_{U_{\lambda}})_*\Ocal_{\lambda}))^\cd$$
Denoting by $\check{C}$ the Čech complex
$$\coprod_{\lambda \in \Lambda_1}(i_{U_{\lambda}})_*\Ocal_{\lambda}\to \cdots \to \coprod_{\lambda \in \Lambda_{n-1}}(i_{U_{\lambda}})_*\Ocal_{\lambda} \to (i_{U_{[n]}})_* \Ocal_{\lambda_{[n]}},$$
we see that $\Hom_\Dcal(M,\check{C}^\bc)$ is naturally isomorphic to $\Gamma(X,\check{C} \otimes_X M)^\cd \cong \Gamma(X,M)^\cd$ as desired, using \cite[Lemma III.4.2]{Har77}.

For a Noetherian, but not necessarily semi-separated scheme $X$, we still get that $\Ocal_X^\bc$ can be represented by a bounded complex of injectives. 

\begin{lemma}\label{dualCech-I}
  Let $X$ be semi-separated or Noetherian. Then the object $\Ocal_X^\bc$ of $\D$ can be represented as a bounded complex of injective quasi-coherent sheaves concentrated in nonpositive degrees.
\end{lemma}
\begin{proof}
  The semi-separated case is done above, so let us assume that $X$ is Noetherian. Let $F \in \Qcoh{X}$ and $i \in \Zbb$. Then we have the isomorphism 
  $$\Hom_\D(\Sigma^i F,\Ocal_X^\bc) \cong \Gamma(X,\Sigma^i F)^\cd.$$ 
  Clearly, $\Gamma(X,\Sigma^i F) = 0$ for any $i < 0$. On the other hand, there is $d>0$ such that $\Gamma(X,\Sigma^i F) = 0$ for any $i > d$, this follows by \cite[071L]{Stacks}. Then we apply \cref{hoinjectives}.
\end{proof}
\subsection{Standard co-$t$-structure}
Let us discuss the standard co-$t$-structure $(\Pcal^{\geq 0},\WDcal{0})$ as introduced more generally in \cref{ss:cotstructres}, that is, the co-$t$-structure generated by $(\D^{\cpt,\leq 0})^*$, where $(-)^*=[-,\Ocal_X]$ is the duality on $\D^\cpt$, so that $\WDcal{0} = (\D^{>0})^\vee$ is the dual definable category to $\D^{>0}$. In the affine case, $\WDcal{0}$ coincides with the standard aisle $\D^{<0}$. This is however not true in general, as indeed, $\D^{<0}$ is not even a definable subcategory of $\D$.

\begin{lemma}\label{WDcal}
  Let $X$ be Noetherian or semi-separated.
  \begin{enumerate}
    \item[(i)] $\WDcal{0} = \{M \in \D \colon [M,\Ocal_X^\bc] \in \D^{>0}\}$.
    \item[(ii)] We have $\D^{\leq -a} \subseteq \WDcal{0} \subseteq \D^{<b}$ for some $a,b\geq 0$. If $X$ is semi-separated, we can take $b=0$.
    \item[(iii)] If $\WDcal{0} = \D^{<0}$ then direct products are exact in $\Qcoh X$. 
    \item[(iv)] If $X$ is affine then $\WDcal{0} = \D^{<0}$.
  \end{enumerate}
\end{lemma}
\begin{proof}
  $(i)\colon$ This follows immediately from dual definability of $\WDcal{0}$ and $\D^{>0}$ and the natural isomorphism $M^\bc \cong [M,\Ocal_X^\bc]$, see \cref{ss:bw-duality}.

  $(ii)\colon$ By \cref{dualCech-I}, $\Ocal_X^\bc \in \D^{> -a}$ for some $a>0$. Then it follows easily from $(i)$ that $\D^{\leq -a} \subseteq \WDcal{0}$. Since $\D^{<0}$ is closed under pure subobjects, it is enough to check $M^\bc \in \D^{<b}$ for any $M \in \D^{>0}$ to check the second inclusion. Now $M^\bc \cong [M,\Ocal_X^\bc]$ and by \cref{dualCech-I}, $\Ocal_X^\bc$ can be realised as a bounded complex of injectives which vanishes in positive degrees. If $X$ is semi-separated, we have by the discussion above that $[M,\Ocal_X^\bc] \cong \SHomqc(M,\Ocal_X^\bc)$, and the latter is clearly a complex which vanishes in nonnegative degrees. In the general situation, we have $[M,\Ocal_X^\bc] \cong \RQ \circ \SHom(M,\Ocal_X^\bc)$. Then we get the bound $b$ from the bound on the cohomological dimension of the derived coherator $\RQ$, see \cite[0CSS]{Stacks}.

  $(iii)\colon$ The equality $\WDcal{0} = \D^{<0}$ implies that $\D^{<0}$ is a definable subcategory of $\D$, in particular, that $\D^{<0}$ is closed in $\D$ under products. This clearly implies that $\Qcoh X$ has exact direct products. 

  $(iv)\colon$ In the affine case, $\Qcoh X \cong \Mod R$ and then it is well-known that $\D^{>0}$ and $\D^{<0}$ are dual definable.
\end{proof}
\begin{rmk}
Under the extra assumption of the existence of an ample family of invertible sheaves, Kanda proves in \cite{Kan19} that exactness of direct products in $\Qcoh{X}$ forces $X$ to be affine. Therefore, \cref{WDcal} yields under the same assumption that $\WDcal{0}=\D^{<0}$ if and only if $X$ is affine.
\end{rmk}

\subsection{Flat and injective dimensions}
We are now ready to discuss how the flat and injective dimensions defined in \cref{ss:dimensions} are related to hoflat and hoinjective resolutions in $\D$. 
\begin{lemma}\label{scheme-dimensions}
The following hold. 
  \begin{enumerate}
    \item[(i)] $M \text{ is isomorphic to a nonnegative hoflat complex} \iff M \in \Fcal^{\geq 0}$
    
    $M \text{ is isomorphic to a bounded below hoflat complex} \iff M \in \Fcal^+$

    \item[(ii )] $M \text{ is isomorphic to a bounded above hoinjective complex} \implies M \in \Ical^{-}$
  
    with the converse implication holding in case $X$ is Noetherian.
  \end{enumerate}
\end{lemma}
\begin{proof}
  $(i)\colon$ Let $M \in \D$ be represented by a nonnegative hoflat complex (in general, with components in $\Mod{X}$) then clearly $M \otimes_X^\mathbf{L} - \cong M \otimes_X -$ preserves $\D^{>0}$, so that $M \in \Fcal^{\geq 0}$. Let us prove the converse. Let $M \in \Fcal^{\geq 0}$ and assume that $M$ is a hoflat complex of flat quasi-coherent sheaves. The condition $M \in \Fcal^{\geq 0}$ implies that $M \otimes_X H \in \D^{\geq 0}$ for any $H \in \Qcoh{X}$. Let $B = B^{0}(M)$ be the $0$-th coboundary sheaf of $M$, so that the truncated complex $\sigma^{\leq -1}M$ is a flat resolution of $B$ in $\Mod{X}$. We see that $B \otimes_X -$ is an exact functor $\Qcoh{X} \to \Mod{X}$. We claim that this implies that $B$ is flat, for which it is enough to check that $B_x$ is a flat $\Ocal_{X,x}$-module for any $x \in X$. We have a natural isomorphism

  $$B_x \otimes_{\Ocal_{X,x}} - \cong (B \otimes_X (i_x)_* -)$$
  which shows that $B_x \otimes_{\Ocal_{X,x}} -$ is left exact, thus exact, establishing the claim. It follows that the soft truncation $M^{\geq 0} = (0 \to B \to M^0 \to M^1 \to \cdots)$ is a complex of flat sheaves. Consider the soft truncation triangle 
  $$M^{<0} \to M \to M^{\geq 0} \rightsquigarrow.$$ 
  The cocycle-coboundary exact sequence $0 \to Z^{-1}(M) \to M^{-1} \to B \to 0$ shows that $Z^{-1}(M)$ is flat, and so $M^{< 0}$ is hoflat. It follows that $M^{\geq 0}$ is hoflat and quasi-isomorphic to $M$, as desired.

  The second equivalence regarding $\Fcal^+$ is a direct consequence of the first one.


  $(ii)\colon$ Let $M$ be a bounded above hoinjective complex. As per the proof of \cref{WDcal}, there is $b > 0$ such that $[-,M] = \RQ \SHom_X(-,M)$ takes $\D^{>0}$ to $\D^{<b}$, and thus also to a suitable shift of $\WDcal{0}$ by \cref{WDcal}. Therefore, $M \in \Ical^{-}$. Assume that $X$ is Noetherian and let us prove the converse, if $M \in \Ical^{-}$ then $[-,M]$ takes $\D^{>0}$ to a suitable shift of $\WDcal{0}$, and therefore to $\D^{<d}$ for some $d>0$ by \cref{WDcal}. Applying $\Gamma(X,-)$, we see that $\Hom_{\D}(\Sigma^i F,M) = 0$ for any $i<-d$ and $F \in \Qcoh{X}$, so that \cref{hoinjectives} applies.
\end{proof}

\section{Classification over Noetherian schemes}
\label{s:classification}

\textit{From now on, $X$ is always a Noetherian scheme.}

\subsection{Local comohology and support}
Let $U$ be an open subset of $X$ and set $V = X \setminus U$. Let 
$$\D_V = \{M \in \D \colon M_x = 0 ~\forall x \in U\},$$ 
this is a compactly generated $\otimes$-ideal of $\D$, see \cite[Corollary 6.8]{BF11}, the separated case comes from \cite[Theorem 5.8]{AJS04}. Let $\Lcal_V \subseteq \Qcoh{X}$ be the localising subcategory of $V$-torsion quasi-coherent sheaves, that is, 
$$\Lcal_V = \{F \in \Qcoh{X} \colon F_x = 0 ~\forall x \in U\},$$ 
and let 
$$t_V\colon \Qcoh{X} \to \Lcal_V$$ 
be the $V$-torsion functor, i.e., the right adjoint to the inclusion $\Lcal_V \hookrightarrow \Qcoh{X}$. Then 
$$\D_V = \{ M \in \D \colon H^i(M) \in \Lcal_V ~\forall i\}$$ 
and the right derived functor $\mathbf{R}t_V\colon \D \to \D_V$ is the right adjoint to the inclusion $\D_V \hookrightarrow \D$. Set $\Gamma_V\colon \D \to \D$ to be the composition of $\mathbf{R}t_V$ with the inclusion $\D_V \hookrightarrow \D$. Following \cite[\S 2.1]{AJS04}, \cite[Lemma 8.12]{Ste13}, we have the triangle
$$\Gamma_V(\Ocal_X) \to \Ocal_X \to \mathbf{R}(i_U)_* (i_U)^* \Ocal_X \rightsquigarrow$$
which is idempotent in the sense of \cite[Definition 3.2]{BF11}. In particular, we have the natural isomorphism $\Gamma_V(-) \cong \Gamma_V(\Ocal_X) \otimes_X^\mathbf{L} -$. Given a point $x \in X$, denote by $\cl{x}$ the closure of $x$ in $X$ and let $\Gamma_x$ denote the composition $\mathbf{R}(i_x)_* \Gamma_{\cl{x}}(-)_x\colon \D \to \D$. As per \cite[\S 8]{BF11}, we have $\Gamma_x(-) \cong \Gamma_x(\Ocal_X) \otimes_X^\mathbf{L} -$. Given any $M \in \D(X)$, define the \newterm{Balmer-Favi support} of $M$ as
$$\supp(M) = \{x \in X \colon \Gamma_x(M) \neq 0\}.$$ 
Let $\Gamma_x \D$ denote the essential image of the functor $\Gamma_x\colon \D \to \D$, or equivalently, 
$$\Gamma_x \D = \{M \in \D \colon \supp(M) \subseteq \{x\}\}.$$ 
Note that the stalk functor $(-)_x$ restricts to an equivalence between $\Gamma_x \D$ and the subcategory 
$$\{M \in \D(\Ocal_{X,x}) \colon M_y = 0 ~\forall y \in \Spec{\Ocal_{X,x}} \setminus \{x\}\}$$ 
of $\D(\Ocal_{X,x})$ consisting of objects supported only on $\{x\}$, that is, objects whose cohomology is a semiartinian $\Ocal_{X,x}$-module. By abuse of notation, we shall sometimes consider $\Gamma_x \D$ as identified with this subcategory of $\D(\Ocal_{X,x})$. It will be also useful at times to view $\Gamma_x\colon \D \to \D$ as the composition of the following functors:
$$\D \xrightarrow{(-)_x} \D(\Ocal_{X,x}) \to \Gamma_x \D \hookrightarrow \D,$$
where the functor $\D(\Ocal_{X,x}) \to \Gamma_x \D$ is the local cohomology functor over the local ring $\Ocal_{X,x}$.
\begin{lemma}\label{Gamma-Fminus}
  For any $x \in X$ we have $\Gamma_x(\Ocal_X) \in \Fcal^{\geq 0}$.
\end{lemma}
\begin{proof}
  Given $M \in \D^{>0}$, we have $\Gamma_x(\Ocal_X) \otimes_X^\mathbf{L} M \cong \Gamma_x(M) = \mathbf{R}(i_x)_* \Gamma_{\cl{x}}(-)_x(M)$, the latter clearly preserving $\D^{>0}$.
\end{proof}
\subsection{Stratification}
The category $\Dcal$ is stratified in the sense of \cite[4.4]{BHS23}, which amounts to the following. Consider the assignment
$$\supp^{-1}\colon \{\text{subsets of $X$}\} \to \{\text{localising $\otimes$-ideals of $\Dcal$}\}$$
which takes a subset $Y$ of $X$ to $\supp^{-1}(Y) \coloneqq \{M \in \Dcal \colon \supp(M) \subseteq Y\}$. This assignment is always injective and we say that $\Dcal$ is \newterm{stratified} if it is also surjective. One of the main thrusts of stratification theory, which allowed to establish stratification results in other various tt-categories, is that the surjectivity of this assignment amounts to checking the following two conditions:

  $\bullet$ The \textit{Local to Global principle:} For any localising $\otimes$-ideal $\Lcal$ we have $M \in \Lcal$ if and only if $\Gamma_x (M) \in \Lcal$ for all $x \in X$.

  $\bullet$ The \textit{Minimality principle:} For any $x \in X$, $\Gamma_x \D$ is a minimal nonzero localising $\otimes$-ideal of $\D$.

The compactly generated $\otimes$-ideals then correspond to those subsets $V$ of $X$ which are specialisation-closed, that is, $x \in V \implies \cl{x} \subseteq V$. It is standard to check that 
$$\supp^{-1}(V) = \Supp^{-1}(V) \coloneqq \{M \in \Dcal \colon \Supp(M) \subseteq V\}$$ 
for any specialisation-closed subset $V$ of $X$. Note that, in this case, $\supp^{-1}(X \setminus V)$ is the \textit{image} of the smashing localisation functor whose kernel is $\supp^{-1}(V)$.

\subsection{Bousfield stratification}\label{ss:Bousfield-stratification}
Note that any localising $\otimes$-ideal of the form $\supp^{-1}(Y)$ is a Bousfield class. Indeed, we have  $\supp^{-1}(Y)=\Bou{\coprod_{x \in X \setminus Y}\Gamma_x(\Ocal_X)}$. In fact, the assignment $\supp^{-1}$ yields a complete lattice isomorphism between the lattice of subsets of $X$ and $\BL{\D}$, this also applies to the general tt-theoretical stratification, see \cite[Theorem 8.8]{BHS23}. Recall that this is not a completely general fact---in the absence of stratification, a localising $\otimes$-ideal can fail to be a Bousfield (or even cohomological Bousfield) class, see \cite[Corollary 4.11]{Ste14}.

Since our aim is to construct a certain analogue of stratification for semi-Bousfield classes, it is useful to see how stratification works in $\D$ if we restrict to Bousfield classes from the get go. In this case, the reformulation of the Minimality principle in terms of Bousfield classes boils down to the ``Tensor nonvanishing'' in $\D$. In fact, this amounts to the validity of the ``$\otimes$-theorem'' in $\D$, see \cite[Remark 7.20]{BF11}.
\begin{lemma}\label{tensor-nonvanishing-classical}
  If $M,N \in \D$ are such that $x \in \supp(M) \cap \supp(N)$ then $x \in \supp(M \otimes_X^\mathbf{L} N)$. In particular, $M \otimes_X^\mathbf{L} N \neq 0$.
\end{lemma}
\begin{proof}
  We give a short proof of this for convenience and also for comparison with \cref{tensor-vanishing}. Let $k=k(x)$ be the residue field sheaf at $x$ and recall that $k(x)$ detects being supported on $x$ in the sense that for any $L \in \D$ we have $x \in \supp(L)$ if and only if $k(x) \otimes_X^\mathbf{L} L \neq 0$, see \cite{Nee92}. So we need to check nonvanishing of 
  $$k(x) \otimes_X^\mathbf{L} (M \otimes_X^\mathbf{L} N) \cong (k(x) \otimes_X^\mathbf{L} M) \otimes_{k(x)} (k(x) \otimes_X^\mathbf{L} N),$$
  but this is clear as the right hand side is simply a graded tensor product of two nonzero graded $k(x)$-vector spaces.
\end{proof}
\begin{rmk}
    Another advantage of restricting the stratification to Bousfield classes is that it might be achievable in cases that the usual stronger formulation of stratification is not. See \cref{sec:appendix-vnr} for discussion of this in the particular case of a commutative von Neumann regular ring.
\end{rmk}
\subsection{Perversities on $X$} Let $\Zbb^* \coloneqq \Zbb \cup \{-\infty,\infty\}$ denote the extended integer line, which we will often consider as a totally ordered set with the obvious order. By a \newterm{perversity} on $X$, we mean any map $p\colon X \to \Zbb^*$. Let $\Perv{X}$ denote the set of all perversities on $X$. Note that $\Perv{X}$ is a complete lattice with respect to the order inherited from $\Zbb^*$ pointwise. The scheme $X$ also has a poset structure with respect to the \newterm{specialisation order} $\rightsquigarrow$ defined by $x \rightsquigarrow y \iff y \in \cl{x}$. A perversity $p$ on $X$ is called \newterm{monotone} if it is an order-preserving map. As  $X$ is Noetherian, the specialisation order on $X$ satisfies both the ascending and the descending chain conditions, i.e., every non-empty subset of $X$ contains a closed point and a generic point---we will use this property freely.

\begin{rmk}
  Our notion of perversity is slightly more general than those encountered in the literature. In \cite{Bez00, AB09}, only integer-valued perversities are considered. In other sources, perversities are always assumed monotone.
\end{rmk}

\subsection{Main assignment}\label{ss:assignment}
Now we define an assignment
$$\Phi\colon \Perv{X} \to \SBL{\D}$$
by setting 
$$\Phi(p) = \SBouBig{\coprod_{x \in X}\Sigma^{p(x)}\Gamma_x(\Ocal_X)}.$$
We first observe that $\Phi$ extends the stratification assignment $\supp^{-1}$. To that aim, we shall make the identification of the lattice of subsets of $X$ with the sublattice of $\Perv{X}$ consisting of those perversities with image inside $\{-\infty,\infty\}$. Our choice of this identification sends a subset $Y$ of $X$ to the perversity $p_Y$ defined as

$$p_Y(x) = \begin{cases} \infty, & x \in Y \\ -\infty, & x \not\in Y. \end{cases}$$
It is then easy to check that 
$$\Phi(p_Y) = \SBouBig{\coprod_{x \in Y}\Sigma^{\infty}\Gamma_x(\Ocal_X)} = \BouBig{\coprod_{x \in Y}\Gamma_x(\Ocal_X)} = \supp^{-1}(X \setminus Y).$$

\subsection{Local and nonlocal depth}
The \newterm{infimum} of $M \in \Dcal$ is defined as $\inf(M) = \inf \{n \in \Zbb \colon H^n(M) \neq 0\}$. Equivalently, $\inf(M) = \inf \{n \in \Zbb \colon M \in \Dcal^{\geq n}\}$. The supremum $\sup(M)$ is defined dually. Given $x \in X$, the \newterm{(local) depth} of $M$ at $x$ is defined as
$$\depth_x(M) \coloneqq \inf \Gamma_x(M).$$
Therefore, the assignment $\Phi$ is equivalently defined via depth:
$$\Phi(p) = \{M \in \Dcal \colon \depth_x(M) > p(x) ~\forall x \in X\}.$$
In the original work \cite{Bez00, AB09} on $t$-structures arising from perversities, a related but in general different assignment is used using the nonlocal depth instead. The \newterm{nonlocal depth} of $M$ at $x$ is defined as
$$\depth(x,M) \coloneqq\inf \Gamma_{\cl{x}}(M).$$
The two assignments coincide for those perversities which we expect to give rise to compactly generated $\otimes$-$t$-structures, that is, the monotone perversities. Let us check that.
\begin{lemma}\label{monotone-perv}
  Let $p$ be a monotone perversity. Then
  $$\Phi(p) = \{M \in \Dcal \colon \depth(x,M) > p(x) ~\forall x \in X\}.$$
  In this case, $\Phi(p)$ is closed under products, and thus constitutes a coaisle of a $t$-structure.
\end{lemma}
\begin{proof}
  This follows from \cite[Propositions 2.10 and 2.11]{FI03}, see also \cite[\S 2.3]{HNS24}. The fact that $\Phi(p)$ is product-closed then follows from \cite[Theorem 2.1]{FI03}. The last claim is then a formal consequence of \cref{product-tstr}, but \cite[Theorem 2.1]{FI03} also directly yields that $\Phi(p)$ is a right orthogonal to a suspension-closed set of compact objects, and thus the coaisle of a compactly generated $t$-structure, see \cite[Appendix 2]{KN13}.
\end{proof}
However, for nonmonotone perversities, our assignment differs from the one of \cite{AB09}, as one cannot expect the local cohomology functors $\Gamma_{\cl{x}}$ to capture the support of an arbitrary subset of $X$. Note here that if $p_Y$ is the perversity corresponding to a subset $Y$ of $X$ defined above, then $p_Y$ is monotone if and only if $Y$ is specialisation-closed.
\subsection{Local to Global principle for semi-Bousfield classes}\label{ss:Local to Global-principle-for-sb}
The good news is that (even cohomological) semi-Bousfield classes admit a straightforward formulation of the Local to Global principle which holds in our setting.
\begin{lemma}\label{local-piece}
  Let $\Ccal = \CSBou{\Ecal} \in \CSBL{\D}$. Then $M \in \Ccal$ if and only if $\Gamma_x(M) \in \Ccal$ for all $x \in X$.
\end{lemma}
\begin{proof}
  Let us first prove that $\Ccal$ is closed under $\Gamma_x$ for any $x \in X$. For any $E \in \Ecal$ we have 
  $$[\Gamma_x(M),E] \cong [\Gamma_x(\Ocal_X),[M,E]].$$ 
  If $M \in \Ccal$ then $[M,E] \in \WDcal{0}$. It thus remains to observe that $\WDcal{0} \cap \PI$ is closed under $[\Gamma_x(\Ocal_X),-]$, which follows from \cref{F0preserve} and the fact that $\Gamma_x(\Ocal_X) \in \Fcal^{\geq 0}$ by \cref{Gamma-Fminus}. Now assume that $\Gamma_x(M) \in \Ccal$ for all $x \in X$ and let us prove that $M \in \Ccal$. We have $[\Gamma_x(\Ocal_X),[M,E]] \cong [\Gamma_x(M),E]$, and so it suffices to show that for any $M \in \PI$ we have $[\Gamma_x(\Ocal_X),M] \in \WDcal{0}$ for all $x \in X$ implies that $M \in \WDcal{0}$. It is enough to check this for $M=N^\cd$, see \cref{BW-duality}, and thus it remains to check that the following implication holds: 
  $$\Gamma_x(N) \in \D^{>0} ~\forall x \in X \implies N \in \D^{>0}.$$ 
  Towards contradiction, let $x \in \Supp(H^i(N))$ be minimal such that there is $i \leq 0$ with 
  $$H^i(N_x) \cong H^i(N)_x \neq 0.$$ 
  By the assumption, we have $N_y \in \D^{>0}$ for any $y \in X$ such that $x \in \cl{y}$ and $x \neq y$. It follows that $(N_x)^{\leq 0}$ is a complex whose cohomology consists of $\Ocal_{X,x}$-modules supported on $\{x\}$. As a consequence, $\Gamma_x(N) \not\in \D^{>0}$, a contradiction.
\end{proof}
\subsection{Substitute for the Minimality principle}\label{ss:substitute-for-minimality-principle} The Local to Global principle we just established reduces the goal of ``semi-Bousfield stratification'' to analysing what happens inside $\Gamma_x \D$ for each $x \in X$. If a semi-Bousfield class arises from a perversity via the assignment $\Phi$, its intersection with $\Gamma_x \D$ is of the simple form $\Gamma_x \Dcal^{\geq n} \coloneqq \Gamma_x \D \cap \D^{\geq n}$ where $n \in \Zbb^*$. Since these subcategories form a chain isomorphic to $\Zbb^*$, it does not make much sense to talk about ``Minimality principle'' in this setup. Instead, we want to formulate a useful criterion for a semi-Bousfield class to be of this form. In fact in the nonregular case, it will turn out that there are semi-Bousfield classes which are not in the image of the assignment $\Phi$.
\begin{lemma}\label{coloc-PI-Iminus}
  Let $x \in X$ and $E \in \D$.
  \begin{enumerate}
    \item[(i)] We have a natural isomorphism $(E^\bc)^x \cong (E_x)^\cd$ in $\D(\Ocal_{X,x})$.
    \item[(ii)] If $E \in \PI$ then $E^x$ is pure-injective in $\D(\Ocal_{X,x})$.
    \item[(iii)] If $E \in \Ical^-$ then $E^x$ is isomorphic to a bounded above hoinjective complex in $\D(\Ocal_{X,x})$.
  \end{enumerate}
\end{lemma}
\begin{proof}
  $(i)\colon$ For any $N \in \D(\Ocal_{X,x})$, we have 
  $$\Hom_{\D(\Ocal_{X,x})}(N,(E^\bc)^x) \cong \Hom_{\D}(\mathbf{R}(i_x)_* N,E^\bc) \cong \Gamma(X,\mathbf{R}(i_x)_* N \otimes_X^\mathbf{L} E)^\cd \cong $$ $$\cong (N \otimes_{\Ocal_{X,x}}^\mathbf{L} E_x)^\cd \cong \Hom_{\D(\Ocal_{X,x})}(N,(E_x)^\cd)$$
  establishing the isomorphism.

  $(ii)\colon$ By passing to direct summands, it is enough to establish this in the case $E = M^\bc$ for some $M \in \D$. Then this follows from $(i)$, as $E^x = (M^\bc)^x \cong (M_x)^\cd$.

  $(iii)\colon$ Since $E \in \Ical^-$, there is by \cref{scheme-dimensions} an integer $a \in \Zbb$ such that $\Hom_\D(\D^{>a},E) = 0$. Let $F \in \Mod{\Ocal_{X,x}}$, then $\Hom_{\D(\Ocal_{X,x})}(\Sigma ^i F,E^x) \cong \Hom_{\D}(\Sigma ^i \mathbf{R}(i_x)_* F,E)$. If $i<-a$ then $\Sigma ^i \mathbf{R}(i_x)_* F \in \D^{>a}$, making the latter Hom group vanish. The rest follows from \cref{hoinjectives}.
\end{proof}
The following is the promised criterion. Note that the condition $(iii)$ bears certain comparison to the Tensor nonvanishing of \cref{tensor-nonvanishing-classical}. We recall the notation $\widehat{\Ecal} = \{S \otimes E \colon S \in \D^{\leq 0,\cpt}, E \in \Ecal\}$ of \cref{ss:cohomologicalsb}.
\begin{lemma}\label{gen-crit}
  Let $\Ccal = \CSBou{\Ecal} \in \CSBL{\D}$ be a cohomological semi-Bousfield class. Then the following are equivalent:
  \begin{enumerate}
    \item[(i)] $\Ccal$ is in the image of $\Phi$,
    \item[(ii)] for each $x \in X$ and each $M \in \Ccal \cap \Gamma_x \D$ we have $\Gamma_x \Dcal^{\geq \inf(M)} \subseteq \Ccal$.
    \item[(iii)]  For any $x \in X$ and $M \in \Ccal \cap \Gamma_x \D$ we have
    $$M \in \Ccal \implies \Hom_{\D(\Ocal_{X,x})}(\RHom_{\Ocal_{X,x}}(k(x),M),E^x)=0 ~\forall E \in \widehat{\Ecal}$$
  \end{enumerate}
\end{lemma}
\begin{proof}
  $(i) \iff (ii)\colon$ The direct implication is clear from the description of $\Phi(p)$, so we need to show the converse. Let us define a perversity $p$ on $X$ by setting $p(x)$ to be the infimum of $\inf(M)$, where $M$ runs through $\Ccal \cap \Gamma_x \Dcal$. We claim that $\Ccal = \Phi(p)$. Since both $\Ccal$ and $\Phi(p)$ are cohomological semi-Bousfield classes, by \cref{local-piece} it is enough to show for any $x \in X$ and $M \in \Gamma_x \Dcal$ that $M \in \Ccal \iff M \in \Phi(p)$, but this equivalence follows directly from the condition $(ii)$.

  $(ii) \iff (iii)$. Without loss of generality, we assume $\Ecal = \widehat{\Ecal}$. For any $M \in \Gamma_x \D \subseteq \D(\Ocal_{X,x})$ we have $\Hom_{\D(\Ocal_{X,x})}(M,E^x)=0 ~\forall E \in \widehat{\Ecal}$ if and only if $M \in \Ccal$. Indeed, we have by \cref{csB-orth} that 
  $$M \in \Ccal \iff ~\forall E \in \Ecal\colon 0= \Hom_\D(M,E) \cong \Hom_{\D(\Ocal_{X,x})}(M,E^x) ~\forall E \in \widehat{\Ecal}$$
  Thus the condition $(iii)$ equivalently states that if $M \in \Ccal$ then $\RHom_{\Ocal_{X,x}}(k(x),M) \in \Ccal$. By \cite[Theorem 2.1]{FI03}, we have $\inf(M) = \inf \RHom_{\Ocal_{X,x}}(k(x),M)$, which shows $(ii) \implies (iii)$. To see the converse, recall that $\RHom_{\Ocal_{X,x}}(k(x),M)$ is isomorphic to a split complex whose cohomology is a vector space over the field $k(x)$. Then $(iii)$ implies that $k(x)[-i] \in \Ccal$ for any $i \geq \inf(M)$. We claim that this implies $N \in \Ccal$ for any $N \in \Gamma_x \D^{\geq \inf(M)}$. The argument goes by the standard deconstruction of $N$ to pieces built from $k(x)$. If $N$ is a stalk complex in some degree $i \geq \inf(M)$, then since $\supp(N) \subseteq \{x\}$, we can build $N$ by taking transfinite extensions of $k(x)[-i]$ inside $\Ccal$, which we can do as $\Ccal$ is closed under direct limits. Since $\Ccal$ is closed under extensions, this immediately extends to $N$ with bounded cohomology. If $N \in \D^+$, we can write $N$ as a homotopy colimit of its right soft truncations $N^{<n}$. Finally, if $N$ is general, we can represent it as an actual complex of $\Ocal_{X,x}$-modules supported on $\{x\}$, and then $N$ is a homotopy colimit of its left brutal truncations $\sigma^{>n}(N)$, each of which belongs to $\Ccal$.
\end{proof}
\begin{rmk}
  The passage to the tensor product $k(x) \otimes -$, followed by a similar deconstruction as in the last proof, played a crucial role already in the original stratification argument for commutative Noetherian rings of Neeman \cite{Nee92}. In the later classification story of $t$-structure, the dual trick of applying the internal Hom functor $[k(x),-]$ to the coaisle side of the $t$-structure proved successful, see e.g. \cite{HN21} or \cite{Hrb20}, and the same condition manifests in \cref{gen-crit}. However, while the product closure of coaisles makes it straightforward to check that they are actually closed under $[k(x),-]$ \cite[Proposition 2.3]{Hrb20}, this is much less clear for semi-Bousfield classes. This is what necessitates the commutative algebra analysis of \cref{ss:cohomology-nonvanishing} that we apply in what follows.
\end{rmk}
\subsection{Main result for regular schemes}
Recall that a Noetherian scheme $X$ is called \newterm{regular} provided that $\Ocal_{X,x}$ is a regular local ring for all $x \in X$. If $X$ is a regular scheme, it turns out that our assignment $\Phi\colon \Perv{X} \to \SBL{\D}$ is not only injective but also surjective, establishing that $\D$ is \newterm{semi-Bousfield stratified}.
\begin{theorem}\label{T:regular}
  The following are equivalent for a Noetherian scheme $X$:
  \begin{enumerate}
    \item[(i)] $\Phi$ induces a bijection $\Perv{X} \to \CSBL{\D}$,
    \item[(ii)] $\Phi$ induces a bijection $\Perv{X} \to \SBL{\D}$,
    \item[(iii)] $X$ is regular. 
  \end{enumerate}
\end{theorem}
\begin{proof}
  $(i) \implies (ii)\colon$ Clear.

  $(ii) \implies (iii)\colon$ We prove the converse implication by contradiction, so let $x \in X$ be such that $\Ocal_{X,x}$ is a singular local ring. Let $k = k(x)$ be the residue field sheaf at $x$. Let $\Ccal = \SBou{k} \cap \Gamma_x \D = \SBou{k,\coprod_{y \neq x}\Sigma^\infty \Gamma_y(\Ocal_X)} \in \SBL{\D}$. We claim that $\Ccal$ is not in the image of $\Phi$. To that aim, observe first that $\Sigma^n k \not\in \Ccal$ for any $n \in \Zbb$. Indeed, the singularity of $\Ocal_{X,x}$ ensures that $k \otimes_X^\mathbf{L} k \cong k \otimes_{\Ocal_{X,x}}^\mathbf{L} k \not\in \D^+$ by \cite[0AVJ]{Stacks}. On the other hand, $\Sigma^{-1}\Gamma_{x}(\Ocal_X)$ belongs to $\Ccal$, as $\Sigma^{-1}\Gamma_{x}(\Ocal_X) \in \Gamma_x \D$ and $\Sigma^{-1}\Gamma_{x}(\Ocal_X) \otimes_R^\mathbf{L} k \cong \Sigma^{-1}k \in \D^{>0}$. Then $\Ccal$ does not belong to the image of $\Phi$ by \cref{gen-crit}.

  $(iii) \implies (i)\colon$  We will check the condition $(iii)$ of \cref{gen-crit} for a cohomological semi-Bousfield class $\Ccal = \CSBou{\Ecal}$. Let $x \in X$, $M \in \Ccal \cap \Gamma_x \Dcal$, and $E \in \widehat{\Ecal}$. Since $X$ is regular we can apply \cref{avramov-foxby} in order to obtain
  \begin{equation}
    \label{eq-sup-RHom}
      \sup \RHom_{\Ocal_{X,x}}(\RHom_{\Ocal_{X,x}}(k,M),E^x) = \sup(k \otimes_{\Ocal_{X,x}}^\mathbf{L}  \RHom_{\Ocal_{X,x}}(M,E^x)).
  \end{equation}
  Since $M \in \Ccal$, we have $[M,E] \in \WDcal{0}$. Since $M \in \Gamma_x \D$, we have $[M,E] \cong \RHom_{\Ocal_{X,x}}(M,E^x)$. Also, we have for any $i\geq 0$ the isomorphism 
  $$0 = \Hom_{\D}(\Sigma^{i}\Ocal_X,[M,E]) \cong \Hom_{\D(\Ocal_{X,x})}(\Sigma^{i}\Ocal_{X,x},\RHom_{\Ocal_{X,x}}(M,E^x)),$$ 
  showing that $\RHom_{\Ocal_{X,x}}(M,E^x) \in \D^{<0}$. Then we also have
  $$k \otimes_{\Ocal_{X,x}}^\mathbf{L}  \RHom_{\Ocal_{X,x}}(M,E^x) \in \D^{<0},$$
  which together with \cref{eq-sup-RHom} yields
  $$\RHom_{\Ocal_{X,x}}(\RHom_{\Ocal_{X,x}}(k,M),E) \in \D^{<0}.$$ 
  This implies the desired vanishing 
  $$\Hom_{\D(\Ocal_{X,x})}(\RHom_{\Ocal_{X,x}}(k,M),E)=0$$ 
  which in view of \cref{gen-crit} verifies that $\Ccal$ belongs to the image of $\Phi$.
  \end{proof}
\subsection{Main result for singular schemes}
\label{ss:main-result}
In view of \cref{T:regular}, we cannot expect $\Phi$ to be a bijection onto $\SBL{\D}$ if $X$ is singular. We can however describe the image of $\Phi$ in a satisfactory way. Given a subcategory $\Ccal$ of $\D$, let $\SBL{\Ccal}$ denote the sublattice of $\SBL{\D}$ consisting of semi-Bousfield classes of the form $\SBou{\Xcal}$ for a subcategory $\Xcal$ of $\Ccal$. 

\begin{theorem}\label{T:main}
  Let $X$ be a Noetherian scheme. Then $\Phi$ induces a bijection 
  $$\Perv{X} \to \SBL{\Fcal^+}.$$
\end{theorem}
\begin{proof}
  By \cref{Gamma-Fminus}, the assignment $\Phi\colon \Perv{X} \to \SBL{\D}$ lands inside $\SBL{\Fcal^+}$. It remains to show that the image is the whole of $\SBL{\Fcal^+}$.

  Let $\Ccal = \CSBou{\Ecal}$, where $\Ecal = \Gcal^\bc$ for a choice of $\Gcal \subseteq \Fcal^+$. Fix $x \in X$ and $M \in \Ccal \cap \Gamma_x \Dcal$, and let us check the equivalent conditions of \cref{gen-crit} for $\Ccal$. Note first that for any $S \in \D^{\leq 0,\cpt}$, the functor $S \otimes_X^\mathbf{L} -$ preserves $\Ical^- \cap \PI$, we can thus assume $\Ecal = \widehat{\Ecal}$. Next, also the functor $(-)^x$ preserves $\Ical^- \cap \PI$ by \cref{coloc-PI-Iminus}. We see that $E^x \in \Ical^-$ for any $E \in \widehat{\Ecal}$. Let $G \in \Gcal$ be such that $E=G^\bc$ so that $E^x = (G_x)^\cd$ by \cref{coloc-PI-Iminus}. Therefore, \cref{tensor-vanishing-dual} yields 
  $$\sup \RHom_{\Ocal_{X,x}}(\RHom_{\Ocal_{X,x}}(k,M),E^x) = \sup(k \otimes_{\Ocal_{X,x}}^\mathbf{L}  \RHom_{\Ocal_{X,x}}(M,E^x)),$$
  and we can thus follow the same argument as in the proof of \cref{T:regular}.
\end{proof}
Let $\Fcal^b = \Fcal^+ \cap \Tcal^-$ and observe that, in view of \cref{scheme-dimensions}, an object $F \in \D$ belongs to $\Fcal^b$ if and only if $F$ can be represented by a bounded hoflat complex. Thus, the class $\Fcal^b$ consists precisely of objects of \newterm{finite Tor dimension} \cite[08FY]{Stacks}. If $X$ is semi-separated, this is further equivalent to $F$ being representable by a bounded complex of flat quasi-coherent sheaves and in that case we also say that $F$ is of \newterm{finite flat dimension}. With this setup, \cref{T:main} restricts to the formulation of \cref{IT:main}.
\begin{cor}
  \label{C:main-finite-tor}
  Let $X$ be a Noetherian scheme. Then $\Phi$ induces a bijection 
  $$\Perv{X} \to \SBL{\Fcal^b}.$$
  In particular, $\SBL{\Fcal^b} = \SBL{\Fcal^+}$.
\end{cor}
\begin{proof}
  The image of the assignment $\Phi$ in fact lands inside $\SBL{\Fcal^b}$ as $\Gamma_x(\Ocal_X) \in \Tcal^-$ for each $x \in X$, see e.g. \cite[02UZ]{Stacks}. The rest is a direct consequence of \cref{T:main}.
\end{proof}
\begin{cor}
  Let $\Xcal$ be a subcategory of $\D$ consisting of bounded above hoflat complexes. Then there is a perversity $p \in \Perv{X}$ such that $\SBou{\Xcal} = \Phi(p)$.
\end{cor}
\begin{proof}
  Combine \cref{T:main} with \cref{scheme-dimensions}.
\end{proof}

\section{Compactly generated $\otimes$-$t$-structures in $\D_{\qc}(X)$}
\label{s:tstructures}
The goal of this section is to show that the restriction of \cref{T:main} to monotone perversities, relatively quickly recovers the classification of compactly generated $\otimes$-$t$-structures achieved in \cite[Theorem 3.11]{AJS10} for the affine case and in \cite[Theorem 4.11]{DS23} for general Noetherian schemes. An exposition of these classification results from the viewpoint of perversity functions is available in \cite{Sah24}. 

\begin{prop}\label{scheme-tstr}
  Let $p$ be a perversity on $X$. The following are equivalent:
  \begin{enumerate}
    \item[(i)] $p$ is monotone,
    \item[(ii)] $\Phi(p)$ is closed under products in $\D$. 
  \end{enumerate}
\end{prop}
\begin{proof}
  $(i) \to (ii)\colon$ This is \cref{monotone-perv}.

  $(ii) \to (i)\colon$ Towards a contradiction, let $x,y \in X$ be points such that $x \rightsquigarrow y$ but $p(y) < p(x)$. Let $\mm$ denote the maximal ideal of $\Ocal_{X,y}$ and $\pp$ denote the prime ideal corresponding to the point $x$. Consider the $\Ocal_{X,y}$-modules $M_i = \Ocal_{X,y}/(\mm^i + \pp)$. Then $\Sigma^{-p(y)-1}M_i \in \Phi(p)$ for all $i>0$, as for any $z \in X$ we have
  $$\Gamma_z(M_i) = \begin{cases} M_i, & z=y  \\ 0, & \text{otherwise}  \end{cases}$$
  making it clear that $\Gamma_z(M_i) \in \D^{>p(y)}$. On the other hand, the product $\prod_{i>0} \Sigma^{-p(y)-1} M_i \not\in \Phi(p)$, as there is a monomorphism $\Ocal_{X,y}/\pp \hookrightarrow \prod_{i>0} M_i$, and thus $\Gamma_x(\prod_{i>0} \Sigma^{-p(y)-1} M_i) \not\in \D^{>p(x)}$.
\end{proof}

\begin{lemma}\label{perfect-approx2}
  For any closed subset $V$ of $X$ there is $S \in \D^{\leq 0,\cpt}$ such that $\supp(S) \subseteq V$ and $\Supp(H^0(S)) = V$.
\end{lemma}
\begin{proof}
  There is a coherent sheaf $E \in \coh{X}$ such that $\Supp(E) = V$. Indeed, we can take $E = \coprod_{i \in \Zbb}H^i(C)$, where $C$ is a compact object satisfying $\supp(C) = V$, which exists e.g. by \cite[Proposition 2.14(b)]{Bal05}. Then $S$ is obtained by applying \cref{perfect-approx} to $E$.
\end{proof}

\begin{theorem}\label{T:tstr}
  The assignment $p \mapsto (\Ucal_p,\Vcal_p)$, where 
  $$\Ucal_p = \{M \in \D \colon \Supp (H^i(M)) \subseteq \{x \in X \colon p(x) \geq i\} ~\forall i \in \Zbb\}$$
  and
  $$\Vcal_p = \Phi(p) = \{M \in \D \colon \Gamma_{\cl{x}}(M) \in \D^{>p(x)} ~\forall x \in X\}$$
  yields a bijection
  $$
\begin{tikzcd}[column sep = 5.0em, row sep = 0.1em]
 \begin{Bmatrix} 
  \text{Monotone perversities} \\ \text{on $X$} \end{Bmatrix} \arrow[leftrightarrow]{r}{1-1} & \begin{Bmatrix} 
    \text{Compactly generated} \\ \text{$\otimes$-$t$-structures in $\D$} \end{Bmatrix} \\
\end{tikzcd}
$$
\end{theorem}
\begin{proof}
  \cref{scheme-tstr} shows that the assignment yields a bijection between monotone perversities and those $\otimes$-$t$-structures whose coaisles are in the image of $\Phi$, while all the compactly generated $\otimes$-$t$-structures are in this image by \cref{cgtstrtotensor} and \cref{compactsflat}. It thus remains to show that the $t$-structure $(\Ucal_p,\Vcal_p)$ is compactly generated for any monotone perversity $p$. 
  
  By taking intersection of shifts of coaisles, it is enough to check that the coaisle of the form $\Phi(p) = \{M \in \D \colon \Gamma_{\cl{x}}(M) \in \D^{>0}\}$ is compactly generated for any $x \in X$, note that here $p$ is the monotone perversity such that $p(y) = 0$ for all $y \in \cl{x}$ and $p(y) = -\infty$ otherwise. In fact, by \cref{cgtstrtotensor}, it is enough to find a compact object $S \in \D^{\cpt}$ such that $\Phi(p)$ is equal to 
  $$\Vcal \coloneqq \{M \in \D \colon [S,M] \in \D^{>0}\}.$$ 
  Let $S$ be an object of $\D^{\leq 0,\cpt}$ satisfying $\supp(S) \subseteq \cl{x}$ and $\Supp(H^0(S)) = \cl{x}$ provided by \cref{perfect-approx2}. Clearly, we have $\Phi(p) \subseteq \Vcal$. By \cref{T:main} and \cref{cgtstrtotensor}, there is a perversity $q$ such that $\Vcal = \Phi(q)$, and the inclusion yields $q \leq p$. It remains to show that $q(x) = 0$; this follows since $H^0(S) \not\in \Phi(q)$.

  The last job is to check the cohomological description of $\Ucal_p$. Set $\Ucal \coloneqq \Ucal_p = \Perp{0}\Vcal_p$ and 
  $$\overline{\Ucal} = \{M \in \D \colon \Supp (H^i(M)) \subseteq \{x \in X \colon p(x) \geq i\} ~\forall i \in \Zbb\}.$$ 
  Our goal is to show $\Ucal = \overline{\Ucal}$. It is easier to show the inclusion $\Ucal \subseteq \overline{\Ucal}$. Indeed, let $M \in \Ucal$, $x \in X$, and consider the canonical map $M_x \to (M_x)^{>p(x)}$. By the above, $(M_x)^{>p(x)}$ belongs to $\D^{>p(x)}(\Ocal_{X,x})$, and thus $(M_x)^{>p(x)} \in \Vcal$. On the other hand, $M_x \in \Ucal$. This immediately yields $(M_x)^{>p(x)} = 0$, showing that $M \in \overline{\Ucal}$.

  The converse implication $M \in \overline{\Ucal} \implies M \in \Ucal$ will be proved in several steps. Let $M$ be represented as a cochain complex with components in $\Qcoh{X}$.

  \textbf{Case 1:} $M \in \Sigma^i\Qcoh{X}$ for some $i \in \Zbb$. Without loss of generality, let us prove the case $i=0$. Then the condition $M \in \overline{\Ucal}$ translates simply as $\Supp(M) \subseteq{V_0}$, where $V_0 = \{x \in X \colon p(x) >0\}$. For any $N \in \Vcal$ we thus have $\Hom_\D(M,N) = \Hom_\D(\Gamma_{V_0}M,N) = \Hom_\D(\Gamma_{V_0}M,\Gamma_{V_0}N) = 0$, the last vanishing provided by $\Gamma_{V_0}N \in \D^{>0}$.

  \textbf{Case 2:} $M$ is cohomologically bounded. This is easily deduced from Case 1, as $\Ucal$ is closed under extensions.

  \textbf{Case 3:} $M \in \D^+$. The soft truncation functors $(-)^{\leq n}$ for $n \geq 0$ yield a directed tower 
  $$M^{\leq 0} \to  M^{\leq 1} \to M^{\leq 2} \to \cdots$$
  where each object $M^{\leq n}$ belongs to $\overline{\Ucal}$ and is cohomologically bounded, and thus to belong $\Ucal$ by Case 2. Therefore, the homotopy colimit $\hocolim_{n\geq 0} M^{\leq n}$ of this tower belongs to $\Ucal$. By \cite[Proposition 1.2.1.19, Remark 1.2.1.20]{Lur17}, the $t$-structure $(\D^{\leq 0},\D^{>0})$ is right-complete, which amounts to $M$ being isomorphic to $\hocolim_{n\geq 0} M^{\leq n}$.

  \textbf{Case 4:} $M$ is a general object in $\overline{\Ucal}$. Consider the truncation triangle $U \to M \to V \rightsquigarrow$ with respect to the $t$-structure $(\Ucal,\Vcal)$, that is, $U \in \Ucal$ and $V \in \Vcal$. Then $V$ belongs to $\Vcal \cap \overline{\Ucal}$. Let $x$ be a generic point of $\supp(V)$. Then $V' \coloneqq \Gamma_x(V) \cong \mathbf{R}(i_x)_* V_x$ also belongs to $\Vcal \cap \overline{\Ucal}$ and $\supp(V') = \{x\}$. If $V'$ belongs to $\D^+$ then $V' \in \Ucal$ by Case 3, which implies $V' = 0$, a contradiction. Therefore, $V' \cong \Gamma_x V' \in \Vcal$ is cohomologically unbounded below, which implies $p(x) = -\infty$. As $p$ is monotone, we have $p(y) = -\infty$ for all $y \in X$ such that $x \in \cl{y}$, and thus $\D(\Ocal_{X,x}) \subseteq \Vcal$. This in turn implies that $\Ucal \cap \Gamma_x \D = 0$. This is absurd, as $V'^{\geq i} \in \Ucal \cap \Gamma_x$ for every $i \in \Zbb$ by Case 3, and there must be $i$ such that $V'^{\geq i} $ is nonzero.
\end{proof}

\subsection{Comonotone monotone perversities}\label{ss:comonotone}
A perversity $p$ on $X$ is \newterm{comonotone} if $p(y) - p(x) \leq 1$ for any \newterm{immediate specialisation} $x \rightsquigarrow y$ in $X$, that is, if  $x \rightsquigarrow z \rightsquigarrow y$ then $z \in \{x,y\}$. Following prior work of Deligne, Bezrukavnikov \cite{AB09} showed that if $X$ admits a dualising complex then the perversities which are both monotone and comonotone give rise to $t$-structures in the bounded derived category $\D^\bdd(\coh{X})$ of coherent sheaves on $X$. The assignment is nothing else than the restriction of the assignment of \cref{T:tstr} to $\D^\bdd(\coh{X})$, see \cite[\S 4]{Sah24}. In the affine case, where $X=\Spec{R}$ for a commutative Noetherian ring $R$, it was conjectured by Stanley \cite[Conjecture 7.10]{Sta10} that this phenomenon is true in general, even in the absence of a dualising complex. Takahashi \cite{Tak23} established this when $R$ is CM-excellent, a common generalisation of rings with dualising complexes and Cohen--Macaulay rings. This was recently extended to the nonaffine setting in \cite{CLMP24}. The conjecture however turns out to be false in full generality, which we use the opportunity to explicitly explain now. By \cite[Corollary 3.12]{AJS10}, any $t$-structure in $\D^\bdd(\coh{X})$ is the restriction of a $t$-structure of the form $(\Ucal_p,\Vcal_p)$ for a monotone perversity $p$. Then \cite[Corollary 5.6]{HM24} asserts that if $R$ admits a codimension function, i.e. perversity $d$ which is both strictly monotone and comonotone, then $(\Ucal_d,\Vcal_d)$ restricts to $\D^\bdd(\coh{X})$ if and only if $R$ is CM-excellent. There are examples of local Noetherian integral domains $R$ which fail to be CM-excellent \cite[Proposition~4.5]{HRW01}, \cite[Example~2.6]{Nis12}, and for such rings the perversity function given by prime ideal height is a codimension function, see the discussion in \cite[Remark 7.21]{HNS24}. Therefore, any such ring $R$ constitutes a counterexample to Stanley's conjecture.

\section{Hereditary Tor-pairs in $\Qcoh{X}$} 
\label{s:Tor-pair}
\textit{In this section, let $X$ be a semi-separated Noetherian scheme.}

\vspace{1em}

Let $\Tor_i^X(-,-) \coloneqq H^{-i}(- \otimes_X^\mathbf{L} -)\colon \Qcoh{X} \times \Qcoh{X} \to \Qcoh{X}$ denote the classical $i$-th derived functor of the tensor product functor on $\Qcoh{X}$. Note that our assumption of $X$ being semi-separated ensures that $\Qcoh{X}$ has enough flat objects, see \cite[Corollary 3.21]{MurA}. A \newterm{hereditary Tor-pair} in $\Qcoh{X}$ is a pair $(\Ccal,\Ecal)$ of subcategories of $\Qcoh X$ such that $\Ccal$ and $\Ecal$ are maximal with respect to the property 
$$\Tor_i^X(C,E) = 0  ~\forall C \in \Ccal ~\forall E \in \Ecal ~\forall i>0.$$ 
Recall from \cref{scheme-dimensions} that $\Fcal^{\geq 0}$ coincides with the subcategory of $\D$ consisting of hoflat complexes of flat quasi-coherent sheaves in nonnegative degrees. We start by showing that hereditary Tor-pairs can be viewed as special kinds of semi-Bousfield classes.

\begin{lemma}\label{Tor-to-sB}
  Given $(\Ccal,\Ecal)$ a hereditary Tor-pair in $\Qcoh{X}$, there is a semi-Bousfield pair in $\D$ of the form
  $$(\Fcal^{> 0} \star \Ccal, \Fcal^{> 1} \star \Sigma^{-1}\Ecal).$$     
\end{lemma}
\begin{proof}
  It is easy to see that $M \otimes_X^\mathbf{L} N \in \D^{>0}$ whenever
  $$M \in \Fcal^{> 1} \cup \Sigma^{-1}\Ecal \text{ ,and } N \in \Fcal^{> 0} \cup \Ccal.$$
  Since semi-Bousfield classes are extension closed, we see that $M \otimes_X^\mathbf{L} N \in \D^{>0}$ for $M \in \Fcal^{> 1} \star \Sigma^{-1}\Ecal$ and $N \in \Fcal^{> 0} \star \Ccal$. Now let $N \in \D$ be such that $N \in \SBou{\Fcal^{> 1} \star \Sigma^{-1}\Ecal}$, and we need to show that $N \in  \Fcal^{> 0} \star \Ccal$. We have $N \in \SBou{\Fcal^{> 1}} = \D^{\geq 0}$. Let us assume that $N$ is a hoflat cochain complex with quasi-coherent flat components and consider the brutal truncation triangle
  $$\sigma^{> 0}N \to N \to \sigma^{\leq 0}N \rightsquigarrow.$$
  Since both $N$ and $\sigma^{\leq 0}N$ are hoflat, so is $\sigma^{> 0}N$ and thus $\sigma^{> 0}N \in \Fcal^{> 0}$. Since the cohomology of $\sigma^{\leq 0}N$ is concentrated in degree zero, it remains to show that $H^0(\sigma^{\leq 0}N) \in \Ccal$. Let $E \in \Ecal$ and consider the induced triangle 
  $$E \otimes_X^\mathbf{L} \sigma^{> 0}N \to E \otimes_X^\mathbf{L} N \to E \otimes_X^\mathbf{L}\sigma^{\leq 0}N \rightsquigarrow$$
  and the induced exact sequence for any $i>0$:
  $$H^{-i}(E \otimes_X^\mathbf{L} N) \to \Tor_i^X(E,H^0(\sigma^{\leq 0}N)) \to H^{-i+1}(E \otimes_X^\mathbf{L} \sigma^{> 0}N).$$
  Since $N \in \SBou{\Sigma^{-1}\Ecal}$, the leftmost term vanishes, while the rightmost term vanishes as $\sigma^{> 0}N$ is hoflat. The other inclusion is proved analogously.
\end{proof}
For example, if we input the trivial hereditary Tor-pair $(\Qcoh{X}, \Flat(X))$, where $\Flat(X)$ is the subcategory of flat quasi-coherent sheaves, we obtain the ``standard'' semi-Bousfield pair $(\D^{\geq 0}, \Fcal^{> 0})$. On can actually characterise the semi-Bousfield pairs which are obtained from a hereditary Tor-pair via this construction.
\begin{lemma}\label{Tor-sB}
  A semi-Bousfield pair $(\Xcal,\Ycal)$ is obtained from a hereditary Tor-pair in the above way if and only if $\Fcal^{> 0} \subseteq \Ycal \subseteq \D^{> 0}$
\end{lemma}
\begin{proof}
  If $(\Xcal,\Ycal)$ is a semi-Bousfield pair, then $\Fcal^{\geq 0} \subseteq \Xcal \subseteq \D^{\geq 0}$ if and only if $\Fcal^{> 0} \subseteq \Ycal \subseteq \D^{> 0}$, and this is the condition we are looking for. Then the hereditary Tor-pair is recovered by putting $\Ccal = \Xcal \cap \D^{\leq 0}$ and $\Ecal = \Ycal \cap \D^{\leq 1}$.  
\end{proof}
The following shows that our result is compatible with, and in fact yields a nonaffine version of, the classification result \cite[Theorem 4.17]{HHLG24}. Let us say that a hereditary Tor-pair $(\Ccal,\Ecal)$ is \newterm{generated} by a subcategory $\Gcal \subseteq \Qcoh{X}$ if 
$$\Ecal = \{M \in \Qcoh{X} \colon \Tor_i^X(G,M) = 0 ~\forall G \in \Gcal ~\forall i>0\}.$$ 
Recall that $M \in \Qcoh{X}$ is of \newterm{finite flat dimension} if there is $i>0$ such that $\Tor_i^X(M,-) = 0$ and that it is \newterm{locally of finite flat dimension} if for each $x \in X$ there is $i>0$ such that $\Tor_i^X(M,-)_x \cong \Tor_i^{\Ocal_{X,x}}(M_x,(-)_x) = 0$. Note by the locality of tensor product, if a hereditary Tor-pair $(\Ccal,\Ecal)$ is generated by a subcategory $\Gcal$ consisting of quasi-coherent sheaves which are locally of finite flat dimension, it is also generated by the subcategory $\{(i_x)_* G_x \colon G \in \Gcal, x \in X\}$ which consists of quasi-coherent sheaves of finite flat dimension.
\begin{prop}\label{P:Tor-pairs}
  The following hold:
  
  \begin{enumerate}
    \item[(i)] Let $p \in \Perv{X}$. Then the semi-Bousfield pair $(\Xcal,\Phi(p))$ is obtained from a hereditary Tor-pair $(\Ccal,\Ecal)$ as in \cref{Tor-to-sB} if and only if $0 \leq p(x) \leq \depth(\Ocal_{X,x})$ for all $x \in X$. In this situation, we have $\Ecal = \{M \in \Qcoh{X} \colon \depth_x(M) \geq p(x) ~\forall x \in X\}$.

    \item[(ii)] Let hereditary Tor-pair $(\Ccal,\Ecal)$ and semi-Bousfield pair $(\Xcal,\Ycal)$ obtained from \cref{Tor-to-sB}. Then the semi-Bousfield class $\Ycal$ is in $\SBL{\Fcal^+}$ if and only if $(\Ccal,\Ecal)$ is generated by a subcategory of quasi-coherent sheaves (locally) of finite flat dimension.
  \end{enumerate}
\end{prop}
\begin{proof}
  $(i)\colon$ First, the condition $\Phi(p) \subseteq \D^{>0}$ is clearly equivalent to $0 \leq p(x)$ for all $x \in X$. In view of \cref{Tor-sB}, it remains to show that $\Fcal^{> 0} \subseteq \Phi(p)$ if and only if $p(x) \leq \depth(\Ocal_{X,x})$ for all $x \in X$. If $\Fcal^{>0} \subseteq \Phi(p)$ then in particular $\Sigma^{-1}\Ocal_X \in \Phi(p)$ and thus $\Gamma_x(\Ocal_X) \in \D^{\geq p(x)}$ for all $x \in X$. As $\inf(\Gamma_x(\Ocal_X)) = \depth(\Ocal_{X,x})$, this forces $p(x) \leq \depth(\Ocal_{X,x})$ for all $x \in X$. The converse also follows, as $p(x) \leq \depth(\Ocal_{X,x})$ shows that $\Gamma_x(\Ocal_X) \in \D^{\geq p(x)}$, and thus $\Gamma_x(F) \cong \Gamma_x(\Ocal_X) \otimes_X F \in \D^{>p(x)}$ for any $F \in \Fcal^{>0}$. It remains to check the final claim about the description of $\Ecal$. We have
  $$\Fcal^{> 1} \star \Sigma^{-1}\Ecal = \Phi(p) = \{M \in \D \colon \depth_x(M) > p(x) ~\forall x \in X\},$$
  from which it follows that $\Ecal = \{M \in \Qcoh{X} \colon \depth_x(M) \geq p(x) ~\forall x \in X\}$.

  $(ii)\colon$ The condition on the hereditary Tor-pair $(\Ccal,\Ecal)$ is, by the construction above, equivalent to every complex $M \in \Xcal = \Fcal^{> 0} \star \Ccal$ to be such that $(i_x)_* M_x \in \Fcal^+$, and consequently $\Ycal$ is the semi-Bousfield class of the subcategory $\{(i_x)_* M_x \colon x \in X, M \in  \Xcal\} \subseteq \Fcal^+$. Conversely, if $\Ycal$ is the semi-Bousfield class of a subcategory $\Scal$ of $\Fcal^+$, then each complex of $\Xcal$ is locally in $\Fcal^+$, and thus quasi-coherent sheaves of $\Ecal$ are locally of finite flat dimension.
\end{proof}

\begin{theorem}\label{T:Tor-pair}
  Let $X$ be a semi-separated Noetherian scheme. The assignment $p \mapsto (\Ccal_p,\Ecal_p)$, where $\Ecal_p = \{M \in \Qcoh{X} \colon \depth_x(M) \geq p(x) ~\forall x \in X\}$ yields the following bijection:
  $$
\begin{tikzcd}[column sep = 5.0em, row sep = 0.1em]
 \begin{Bmatrix} 
  \text{Perversities $p$ on $X$} \\ \text{which for all $x \in X$ satisfy} \\ \text{$0 \leq p(x) \leq \depth(\Ocal_{X,x})$} \end{Bmatrix} \arrow[leftrightarrow]{r}{1-1} & \begin{Bmatrix} 
    \text{Hereditary Tor-pairs in $\Qcoh{X}$} \\ \text{generated by sheaves of} \\ \text{finite flat dimension} \end{Bmatrix} \\
\end{tikzcd}
$$
If $X$ is regular, all hereditary Tor-pairs in $\Qcoh{X}$ are in the image of this assignment.
\end{theorem}
\begin{proof}
  Follows directly from \cref{P:Tor-pairs}, \cref{T:main}, and \cref{T:regular}.
\end{proof}
\appendix
\section{Von Neumann Regular rings and continuous perversities}
\label{sec:appendix-vnr}
In this appendix, we take a brief detour and consider the nature of semi-Bousfield classes outside of the Noetherian setting. For this section $R$ will always be a commutative ring and $\Spec{R}$ the prime spectrum of $R$, and as we are working in the affine case, we let $\D(R)=
\D(\Mod{R})$ or sometimes simply $\D$ denote the unbounded derived category of $R$-modules. Additionally, as we are in the affine case, all $t$-structures are $\otimes$-$t$-structures with respect to the standard $t$-structure of $\D$.

For our non-Noetherian study, we have chosen the derived category of a commutative von Neumann regular ring (sometimes called an absolutely flat ring). It is primarly due to work of Stevenson in \cite{Ste14} that, relatively speaking, we have some grasp of the inner workings of its derived category. We now provide some background for what we mean by this.

For a moment, let $R$ be an arbitrary commutative ring. Even outside of the Noetherian case, the collection of compact objects of its unbounded derived category is well understood in a specific sense which we make precise now. 
There is a generalisation of the Hopkins-Neeman result over commutative Noetherian rings mentioned in the introduction: For an arbitrary commutative ring,
there is a lattice isomorphism between the thick subcategories of perfect complexes $\D(R)^c$ and the Thomason subsets of the prime spectrum of $R$ \cite{Tho97}, see \cref{ss:continuous-perv} for details. 

However, the big tt-category $\D$ can fail to be stratified in the usual sense of $\Spec R$ being the classifying space for all localising subcategories. When $\D$ is stratified for an arbitrary commutative ring is an open question, nevertheless, when $R$ is von Neumann regular, we have a full classification:
 $\D$ is stratified if and only if the Local to Global principle holds for localising ideals if and only if $\Spec{R}$ is semi-artinian, \cite[Theorem 6.3]{Ste17} and \cite[Lemma 4.6, Theorem 4.8]{Ste14}. On the way to this result, Stevenson shows that $\D$ is Bousfield stratified, see also \cite[Example 4.21]{BHSZ26}. We will make this explicit here for ease of comparison to the semi-Bousfield case.

The favourable property of working with von Neumann regular rings is that every module is flat, so in particular, the acyclicisation functor at a prime $\mathfrak{p}$ is $t$-exact and isomorphic to the (non-derived) tensor product with the residue field of $R$ at the prime $\mathfrak{p}$, or equivalently, the localisation at the prime $\mathfrak{p}$,  $\Gamma_\mathfrak{p}(-)\cong(-\otimes_R k(\mathfrak{p}))\cong (-\otimes_R R_\mathfrak{p})$. See the beginning of \cite[Section 4]{Ste14} for more details. That is:
\begin{itemize}
    \item Bousfield classes satisfy the \textit{Local to Global principle}: $M \in \Bou{X}$ if and only if $M_\mathfrak{p} \in\Bou{X}$ for every prime ideal $\mathfrak{p}$, 
    \item Bousfield classes satisfy the \textit{Minimality principle}: $\Bou{\coprod_{\mathfrak{q}\in \Spec{R}\setminus \{\mathfrak{p}\}}k(\mathfrak{q})}$ is a minimal Bousfield class, as it coincides with $\Loc(k(\mathfrak{p}))$.
\end{itemize}

We now claim that the derived category of a commutative von Neumann regular ring is always semi-Bousfield stratified. The Local to Global follows exactly as the Bousfield case above. As for the Minimality principle for semi-Bousfield classes as discussed in \cref{ss:substitute-for-minimality-principle}, the essential ingredient is that the acyclicisation functors $\Gamma_\mathfrak{p}$ for $\pp \in \Spec{R}$, which not only detect vanishing in $\D$, but are also $t$-exact with respect to the standard t-structure.

Similarly to \cref{s:classification}, let $\Perv{\Spec{R}}$ denote the lattice of functions on $\Spec R$ with values in $\Zbb^* = \Zbb \cup \{-\infty,\infty\}$.

\begin{prop}\label{p:sBous-vnregular}
    The unbounded derived category of modules over a commutative von Neumann regular ring $R$ is semi-Bousfield stratified. In other words,
    \begin{itemize}
        \item Semi-Bousfield classes satisfy the \textit{Local to Global principle}: $M \in \SBou{X}$ if and only if $M_\mathfrak{p} \in\SBou{X}$ for every prime ideal $\mathfrak{p}$,
        \item Semi-Bousfield classes satisfy the \textit{Minimality principle}: if $M \in \SBou{X}\cap \Gamma_\pp\D$, then  $\Gamma_\pp\D^{\geq \inf (M)} \subseteq \SBou{X}$.
    \end{itemize}
    Therefore, there is a lattice isomorphism 
    \[\Phi \colon \Perv{\Spec{R}}\to \SBL{\D(R)},\]
    given by the assignment
    \[f\mapsto \SBouBig{\coprod_{\pp \in \Spec{R}}\Sigma^{f(\pp)}\Gamma_\pp(R)}.\]
\end{prop}

\begin{proof}
    Let $M \in \D$ be such that $M_\pp \in \SBou{X}$ for all $\pp \in \Spec R$. This means that 
    $$(M \otimes_R^\mathbf{L} X)_\pp \cong (M_\pp \otimes_R^\mathbf{L} X) \in \D^{>0},$$
    which clearly yields $M \otimes_R^\mathbf{L} X \in \D^{>0}$, verifying that $M \in \SBou{X}$.
    
    Now we consider the Minimality principle. Fix $X \in 
    D(R)$ and suppose $M \in \SBou{X}\cap \Gamma_\pp\D$. By \cite[Lemma 4.5]{Ste14} mentioned above, $M$ is a coproduct of suspensions of $k(\pp)$. In particular, $\Sigma^{-\inf (M)}\Gamma_\pp(R) \cong \Sigma^{-\inf (M)}k(\pp)$ is a direct summand of $M$, and as semi-Bousfield classes are closed under direct summands,  $\Sigma^{-\inf (M)}k(\pp) \in \SBou{X}\cap \Gamma_\pp\D$. In turn, semi-Bousfield classes are closed under nonpositive suspensions and coproducts, so we can conclude $\Gamma_\pp\D^{\geq \inf (M)} \subseteq \SBou{X}$.  
\end{proof}

Note that a von Neumann regular ring $R$ is Noetherian if and only if it is semisimple artinian, that is, a finite product of fields, and in this case, $\D$ is semi-Bousfield stratified by \cref{T:regular}. 

\subsection{Continuous perversities}\label{ss:continuous-perv}

We One could ask whether we can distinguish the semi-Bousfield classes which correspond to a coaisle of a compactly generated $t$-structure via the assignment $\Phi$, as occurs in the Noetherian case by \cref{cgtstrtotensor} and \cref{scheme-tstr}. Fortunately, all the material to do so already exists in the literature. First, the classification of compactly generated $t$-structures over commutative Noetherian rings in \cite{AJS10} was extended to arbitrary commutative rings \cite[Theorem 5.6]{Hrb20}. Parallel to the classification of thick subcategories of the perfect complexes $\D^c$, one must replace ``specialisation closed subset'' with ``Thomason subset,'' of the prime spectrum to extend to the non-Noetherian situation. Therefore, unlike in the Noetherian situation, it is not sufficient to simply restrict to perversities which are monotone with respect to the inclusion of prime ideals to ensure closure under products of the image under $\Phi$. We first need to make some topological considerations.

Recall that a topological space is \newterm{Alexandrov} if its open sets are closed under arbitrary intersections or, equivalently, if its open sets are precisely the upper subsets with respect to the induced specialisation order on points. For example, the Alexandrov topology on $\Zbb^*$ has as open subsets the upper subsets with respect to its total order. 

Let $R$ be a general commutative ring. A subset $V$ of $\Spec R$ is called \newterm{Thomason} if it is a union of closed subsets of $\Spec R$ with quasi-compact complements. Moreover, by a \newterm{Thomason filtration} of $\Spec R$ we mean a decreasing sequence $(V_n \mid n \in \Zbb)$ of Thomason subsets of $\Spec R$. In fact, the Thomason subsets form the set of open subsets of the \newterm{Hoschter dual topology} on $\Spec R$, see e.g. \cite[Remark 5.11]{BF11}. Let us say that a function $f\colon \Spec R \to \Zbb^*$ is a \newterm{continuous perversity} if it is continuous with respect to the Hochster dual topology on $\Spec R$ and the Alexandrov topology on $\Zbb^*$. This allows us to reformulate \cite[Theorem 5.6]{Hrb20} in terms of continuous perversities.

\begin{theorem}
    Let $R$ be a commutative ring. Then:
    \begin{enumerate}
        \item[(i)] The assignment $f \mapsto (V_n \mid n \in \Zbb)$ given by $V_n \coloneqq \bigcup_{i\geq n}f^{-1}(i)$ provides a bijection between continuous perversities and Thomason filtrations of $\Spec R$.
        \item[(ii)] The assignment $$f \mapsto \Vcal_f =\{M \in \D(R) \mid \Gamma_{V_n}(M) \in \D^{>n} ~\forall n \in \Zbb\},$$ 
        induces a bijection between continuous perversities $f$ on $\Spec R$ and compactly generated $t$-structures $(\Ucal_f,\Vcal_f)$ on $\D(R)$. Here, $\Gamma_{V_n}\colon \D(R) \to \D(R)$ is the acyclisation with respect to the Thomason subset $V_n = \bigcup_{i\geq n}f^{-1}(i)$ of $\Spec R$.        
    \end{enumerate}
\end{theorem}
\begin{proof}
    $(i)\colon$ The assignment is well defined, as if $f$ is continuous then $\bigcup_{i\geq n}f^{-1}(i)$ is a union of Thomason subsets, and thus is itself a Thomason subset. Clearly, the assignment is also injective. To a given Thomason filtration $(V_n \mid n \in \Zbb)$, we can associated a function $g\colon \Spec R \to \Zbb^*$ by setting $g(\pp) = \sup \{n \in \Zbb \mid \pp \in V_n)$. It is easy to see that $\bigcup_{i\geq n}g^{-1}(i) = V_n$ for each $n \in \Zbb$, which verifies both that $g$ is a continuous perversity and that our assignment is also surjective.

    $(ii)\colon$ Follows by combining the bijections provided by $(i)$ and \cite[Theorem 5.6]{Hrb20}, see also the description of the associated coaisle in \cite[Proposition 5.10]{Hrb20}.
\end{proof}
\begin{rmk}
If $R$ is Noetherian (or, more generally, if $\Spec R$ is a Noetherian space) then Thomason subsets coincide with specialisation closed subsets (\cite[Proposition 7.13]{BF11}), and then continuous perversities coincide with monotone perversities. Indeed, if $\Spec R$ is Noetherian then the Hochster dual topology is an Alexandrov topology.  A map between two Alexandrov topological spaces is continuous if and only if it is order-preserving with respect to the induced specialisation orders. In this way, the continuous assumption on a perversity $f\colon \Spec R \to \Zbb^*$ naturally extends the monotone assumption to non-Noetherian contexts.
\end{rmk}

Now let us return to the case of $R$ von Neumann regular. Then $\Spec{R}$ is a totally disconnected Hausdorff and zero-dimensional, and Thomason subsets of $\Spec{R}$ coincide with the open subsets, in other words, the Hoschter dual topology on $\Spec R$ coincides with the Zariski topology, see the discussion in \cite[Section 3]{Ste14} for a summary. Thus, a continuous perversity is just a continuous function $\Spec R \to \Zbb^*$ with respect to the usual Zariski topology on $\Spec R$. Furthermore, a semi-Bousfield class in $\D(R)$ is closed under all products if and only if it is a coaisle of a compactly generated t-structure. Indeed, one implication follows from the general \cref{product-tstr}, while the converse is specific to this setting and follows from \cite[Corollary 3.12]{BH21}. We show that the classification of compactly generated t-structures in $\D(R)$ can be recovered from \cref{p:sBous-vnregular} in a similar fashion as in the Noetherian case of \cref{s:tstructures}.

\begin{cor}
    Let $R$ be a von Neumann regular commutative ring. In the lattice isomorphism $\Phi$ of \Cref{p:sBous-vnregular}, coaisles of compactly generated $t$-structures in $\SBL{\D}$ correspond to continuous perversites in $\Perv{\D(R)}$.
\end{cor}

\begin{proof}
  In view of the discussion above, we need to show that given a perversity $f$ on $\Spec R$ the corresponding semi-Bousfield class 
  $$\Ccal \coloneqq \SBouBig{\coprod_{\pp \in \Spec{R}}\Sigma^{f(\pp)}\Gamma_\pp} = \{M \in \D(R) \mid M_\pp \in \D^{>f(\pp)} ~ \forall \pp \in \Spec R\}$$
  is closed under products if and only if $f$ is continuous.

  Assume first that $f$ is not continuous. This means that there is $\pp \in \Spec R$ such that every open neighbourhood $\pp \in U \subseteq \Spec R$ contains $\qq_U$ such that $f(\qq_U) < f(\pp)$. Note that this implies that $\pp$ belongs to the closure $\cl{\qq_U \mid U \in \Ucal}$. Let $\Ucal$ be the collection of all neighbourhoods of $\pp$ and fix the $\qq_U$ as in the previous sentence for each $U \in \Ucal$. Set $N = \prod_{U\in \Ucal}k(\qq_U)$ to be the product of all the associated residue fields. Set $J = \bigcap_{U\in \Ucal}\qq_U$ and note that $R/J$ embeds into $N$. Since $V(J) = \{\qq \in \Spec R \mid J \subseteq \qq\}$ is a closed subset of $\Spec R$ containing $\qq_U$ for every $U \in \Ucal$, necessarily $\pp$ belongs to $V(J)$ as well. Then $k(\pp) \otimes R/J \neq 0$, and using flatness of $k(\pp)$, it also follows that $k(\pp) \otimes_R N \neq 0$. By our assumption, $\Sigma^{f(\pp)}k(\qq_U) \in \Ccal$ for each $\qq_U$. On the other hand, the previous computation shows that $H^0 \Gamma_\pp(N) \cong k(\pp) \otimes_R N \neq 0$, and thus $\Sigma^{f(\pp)}N \not\in \Ccal$, showing that $\Ccal$ is not closed under products.

  Conversely, assume that $f$ is continuous. Then there is for each $\pp \in \Spec R$ an open neighbourhood $U_\pp$ such that we can write
  $$\Ccal = \{M \in \D(R) \mid M_\qq \in \D^{>f(\pp)} ~ \forall \pp \in \Spec R, \qq \in U_\pp\}.$$
  Let $J(\pp)$ be an ideal of $R$ such that $U_\pp = \Spec R \setminus V(J(\pp))$. Then the condition $M_\qq \in \D^{>f(\pp)} ~ \forall \qq \in U_\pp$ is equivalent to $H^i(M)$ being annihilated by $J(\pp)$ for all $i \leq f(\pp)$, recall here that any ideal of $R$ is idempotent. This condition is clearly stable under products.
\end{proof}

\section{Tensor nonvanishing}
\label{ss:cohomology-nonvanishing}
In this appendix, we inspect certain cohomological (non)vanishing formulas which serve as an essential tool in our characterisation of the image of the assignment $\Phi$ of \Cref{s:classification}. Note that the more or less classical \cref{avramov-foxby} is the only input necessary from this appendix for the regular case of a Noetherian scheme in \cref{T:regular}, the compactly generated case in \cref{T:tstr}, and the Tor-pair case in \cref{T:Tor-pair}. The rest of this appendix deals with  unbounded complexes in the general formulation of \cref{T:main}.

Throughout, let $(R,\mm,k)$ be a local commutative Noetherian ring with derived category $\D = \D(\Mod{R})$. 

\begin{prop}\label{avramov-foxby}
  Let $M,F,E \in \D$ and consider the natural morphisms
  $$\omega\colon F \otimes_R^\mathbf{L} \RHom_R(k,M) \to \RHom_R(k,M \otimes_R^\mathbf{L} F) $$
  and
  $$\theta\colon k \otimes_R^\mathbf{L} \RHom_R(M,E) \to \RHom_R(\RHom_R(k,M),E).$$
  Then $\omega$ and $\theta$ are isomorphisms whenever one of the following conditions holds:
  \begin{enumerate}
    \item[(i)] $R$ is regular, or
    \item[(ii)] $F$ is compact and $E=F^\cd$,
    \item[(iii)] $M \in \D^+$, $F \in \Fcal^+$, and $E \in \Ical^-$.
  \end{enumerate}
\end{prop}
\begin{proof}
  $(i)\colon$ If $R$ is regular then $k$ is a compact object of $\D$. Since both $\omega$ and $\theta$ are easily seen to be isomorphisms if we replace $k$ by $R$, the claim follows from $k \in \thick{R}$. 
  
  $(ii)\colon$ This is proved similarly to $(i)$.

  $(iii)\colon$ This is essentially \cite[Lemma 4.4]{AF91}. For convenience, and also because \cite[Lemma 4.4(I)]{AF91} is stated under slightly more restrictive assumptions, we provide a short proof. We may assume that $M$ is a bounded below complex of injectives, $F$ is a bounded below hoflat complex, and $I$ is a bounded above hoinjective complex. Then $M \otimes_R^\mathbf{L} F \cong M \otimes_R F$ is a bounded below complex of injectives, and thus is hoinjective. Similarly, $\RHom_R(M,E) \cong \Hom_R(M,E)$ is a bounded above complex of flats, and thus is hoflat. The maps $\omega$ and $\theta$ are thus represented by the following natural cochain maps:
  $$\omega\colon F \otimes_R \Hom_R(k,M) \to \Hom_R(k,M \otimes_R F)$$
  and
  $$\theta\colon k \otimes_R \Hom_R(M,E) \to \Hom_R(\Hom_R(k,M),E),$$
  both of which are easily checked to be isomorphisms using the fact that $k$ is a finitely presented $R$-module.
\end{proof}

\begin{rmk}
  \label{avramov-foxby-unbounded-obstruction}
  If $R$ is singular, then the assumption $M \in \D^+$ is essential in the formulation of \Cref{avramov-foxby}(iii). Indeed, setting $F$ to be the stalk complex of the free $R$-module $R^{(I)}$ whose basis is a set $I$, the map $\omega$ is nothing but the natural map
  $$\RHom_{\Ocal_{X,x}}(k,M)^{(I)} \to \RHom_{\Ocal_{X,x}}(k,M^{(I)}) .$$
  We claim that this map being isomorphism, for all $M \in \D$ and any index set $I$, forces $k$ to be compact in $\D$, and thus $R$ to be regular. Let $(M_i)_{i \in I}$ be a collection of objects in $\D$ and set $M = \coprod_{i \in I}M_i$. For each $i \in I$ consider the split inclusion $\iota_i\colon M_i \hookrightarrow M$. Taking the coproduct of $\iota_i$'s over $I$, we obtain a split monomorphism $\iota\colon \coprod_{i \in I}M_i = M \hookrightarrow M^{(I)}$. Applying the natural isomorphism of our assumption to $\iota$ yields that the natural map
  $$\coprod_{i \in I}\RHom_{R}(k,M_i) \to \RHom_{R}(k,\coprod_{i \in I}M_i) $$
  is an isomorphism, showing that $k$ is compact, which in turn forces $R$ to be regular.
\end{rmk}
Although \Cref{avramov-foxby-unbounded-obstruction} shows that \Cref{avramov-foxby}(iii) cannot be directly extended to an unbounded complex $M$, in what follows we show that a weaker version exists, which asserts that the cohomological infima of the two complexes connected by the map $\omega$ coincide even for general unbounded $M$. The following is the main technical input.
\begin{lemma}
  \label{flat-unbounded}
  Let $M,F$ be two complexes in $\D$ such that $\mm \in \supp(M) \cap \supp(F)$. Assume that $M \in \D \setminus \D^+$ and $F \in \mathcal{F}^+$. Then $M \otimes_R^\mathbf{L} F \not\in \D^+$.
\end{lemma}
\begin{proof}
  Let 
  $$\Lcal = \{M \in \Mod{R} \colon \Supp(M) \subseteq \{\mm\}\}$$
  be the Grothendieck subcategory consisting of semi-artinian $R$-modules, i.e. modules with (possibly infinite) composition series. Note that $\D(\Lcal) \cong \Gamma_\mm \D$ is compactly generated, and so we can make use of Krause's results of \cite{Kra05}. Namely, the Verdier localisation functor 
  $$Q\colon \K(\Inj{\Lcal}) \to \D(\Lcal)$$ 
  from the homotopy category of complexes whose components are injective objects of $\Lcal$ to the derived category $\D(\Lcal)$ admits not only the right adjoint $Q_\rho$---which computes hoinjective resolutions---but also the left adjoint $Q_\lambda$. Recall that the category $\Inj{\Lcal}$ of injectives in $\Lcal$ consists precisely of coproducts of copies of the indecomposable injective $R$-module $E(k)$.

  We first reduce the claim to a situation in which both the complexes $M$ and $F$ take a particular shape. We will start by replacing $M$ by $\Gamma_\mm M$. This is legal, as if we show that $\Gamma_\mm M \otimes_R^\mathbf{L} F \cong \Gamma_\mm(M \otimes_R^\mathbf{L} F) \not\in \D^+$ then $M \otimes_R^\mathbf{L} F \not\in \D^+$ follows as a consequence. Then $\supp(M) = \{\mm\}$ and so we can view $M$ as an object in $\D(\Lcal)$. Next, we replace $M$ by the hoinjective complex $Q_\rho(M)$. Explicitly, this means that we assume that the components $M^i$ belong to $\Inj{\Lcal}$ and that $M$ is right orthogonal in $\K(\Inj{\Lcal})$ to any acyclic complex. Finally, by \cite[Appendix B]{Kra05} we may in addition assume that $M$ is minimal, which means that the $i$-th differential $d^i_M$ of $M$ kills the socle of $M^i$ for all $i \in \Zbb$.
  
  Next we adjust the complex $F$. Since $F \in \mathcal{F}^+$, we first represent $F$ by a bounded below hoflat complex so that $F \otimes_R^\mathbf{L} M$ is represented by $F \otimes_R M$. We may in fact assume that the components $F^i$ of $F$ are free $R$-modules. Indeed, $\dim(R) < \infty$ implies that every flat $R$-module has projective dimension at most $\dim(R)$ by \cite[Corollaire (3.2.7)]{RG71} (following \cite[Proposition 6]{Jen70}), and so there is a quasi-isomorphism $F' \to F$ where $F'$ is a bounded below complex of free $R$-modules. The mapping cone of this map is pure-acyclic by \cite[Corollary A.4]{Sha23}, and so $F'$ is hoflat and we may take $F'$ to be our $F$. Denote by $\widehat{F}$ the complex obtained from $F$ by applying the $\mm$-adic completion. Recall that for any free $R$-module $R^{(\lambda)}$, the natural map $R^{(\lambda)} \to \widehat{R^{(\lambda)}}$ becomes an isomorphism after applying $- \otimes_R L$ for any $L \in \Lcal$. It follows that $F \otimes_R M \cong \widehat{F} \otimes_R M$. We can thus replace $F$ by $\widehat{F}$, and the goal becomes showing that $F \otimes_R M \not\in \D^+$ (we do not claim that $\widehat{F}$ is hoflat). Using \cite[Lemma 1.5]{NT20}, we may and will also assume that $k \otimes_R F$ has zero differentials. Explicitly, the components $F^i$ are of the form $\widehat{R^{(\lambda_i)}}$ for suitable cardinals $\lambda_i$ and the differential $d^i_F$ is then given by a $(\lambda_i \times \lambda_{i+1})$ matrix $A_i$ with inputs in $\mm\widehat{R}$ whose columns converge to zero in the $\mm$-adic topology, i.e., for each $n>0$ only finitely many inputs of each column lie outside of the ideal $\mm^n \widehat{R}$, see e.g. \cite[\S 5]{PS22}. As a consequence, the induced morphism 
  $$M^j \otimes (\widehat{R^{(\lambda_i)}} \xrightarrow{A_i} \widehat{R^{(\lambda_{i+1})}}) \cong ((M^j)^{(\lambda_i)} \xrightarrow{A_i} (M^j)^{(\lambda_{i+1})})$$ 
  kills the socle as well. Without loss of generality, we will assume that $F^0 \neq 0$ and $F^i = 0$ for all $i<0$. 
  
  Let us show that $M \otimes_R F \not\in \D^+$. Consider the brutal truncation triangle
  $$F \to F^0 \xrightarrow{h} F' \to \Sigma F.$$
  Fix $l \in \Zbb$ such that $M^l \neq 0$ and consider a nonzero morphism $k \to M^l$, any such morphism has image in the socle of $M^l$. Since the differentials $d_M^i$ kill socles, this induces a chain map $\Sigma^{-l}k \to M$ which is not null-homotopic. Set $P = Q_\lambda(k) \in K(\Inj{\Lcal})$ and recall that there is the coreflection $P \to k$ in $\K(\Lcal)$ whose cone is acyclic. As $M$ is hoinjective as a complex over $\Lcal$, $f$ is homotopy equivalent to the induced map $f\colon \Sigma^{-l}P \to M$ which factorises through the coreflection $\Sigma^{-l}P \to \Sigma^{-l}k$. Embedding $M$ as a summand of $F^0 \otimes_R M \cong M^{(\lambda_0)}$, we can view $f$ as a nonzero element of $\Hom_{\K(\Inj{\Lcal})}(P,\Sigma^{l} F^0 \otimes_R M)$. The composition $h \circ f$ with the connecting morphism $h\colon F^0 \to F'$ is however zero, as its $l$-th component is the composition 
  $$P^0 \to k \to M^0 \to (M^0)^{(\lambda_0)} \xrightarrow{A_0} (M^0)^{(\lambda_1)} \to \bigoplus_{i \geq 0}(M^i)^{(\lambda_{i+1})}$$
  which vanishes, as the map induced by the matrix $A_0$ kills the socle of $(M^0)^{(\lambda_0)}$, and all of the other components are zero. This shows that the induced map
  $$\Hom_{\K(\Inj{\Lcal})}(P,h)\colon \Hom_{\K(\Inj{\Lcal})}(P,\Sigma^{l} F^0 \otimes_R M) \to \Hom_{\K(\Inj{\Lcal})}(P,\Sigma^{l}F' \otimes_R M)$$ 
  is not injective, so that 
  $$\Hom_{\K(\Inj{\Lcal})}(P,\Sigma^{l} F \otimes_R M) \cong \Hom_{\D(\Lcal)}(k,\Sigma^{l} F \otimes_R M) \cong \Hom_{\D(R)}(k,\Sigma^{l} F \otimes_R M)$$ 
  is nonzero. The assumption $M \not\in \D^+$ together with minimality of $M$ ensures that $M^l \neq 0$ for infinitely many $l \ll 0$. Then our argument yields $\RHom_R(k,F \otimes_R M) \not\in \D^+$, which in turn implies the desired conclusion $F \otimes_R M \not\in \D^+$.
\end{proof}

\begin{prop}\label{tensor-vanishing}
  Let $M \in \D$ and $F \in \Fcal^+$. Then the following holds:
  $$\inf F \otimes_R^\mathbf{L} \RHom_R(k,M) = \inf \RHom_R(k,M \otimes_R^\mathbf{L} F). $$
\end{prop}
\begin{proof}
  This is a direct consequence \Cref{avramov-foxby}(iii) if $M \in \D^+$. Thus, we further assume that $M \in \D \setminus \D^+$. If $\mm \not\in \supp(M)$ or $\mm \not\in \supp(F)$ then one can check directly that both $F \otimes_R^\mathbf{L} \RHom_R(k,M)$ and $\RHom_R(k,M \otimes_R^\mathbf{L} F)$ vanish. We can thus assume $\mm \in \supp(M) \cap \supp(F)$, so that \cref{flat-unbounded} applies and yields 
  $$\inf M \otimes_R^\mathbf{L} F = -\infty.$$ 
  By \cite[Theorem 2.1]{FI03}, we have
  $$-\infty = \inf M = \inf \RHom_R(k,M)$$
  and
  $$-\infty = \inf \RHom_R(k,M \otimes_R^\mathbf{L} F).$$
  Since $\RHom_R(k,M)$ can be represented by a graded $k$-vector space which is not bounded below, we easily also check the equality
  $$\inf F \otimes_R^\mathbf{L} \RHom_R(k,M) = -\infty$$
  which concludes the proof.
\end{proof}
As a consequence, we also obtain a dual statement which will be convenient for our application in \cref{T:main}. We \textit{do not know} if the following assumption can be relaxed to just $E \in \Ical^-$.
\begin{prop}\label{tensor-vanishing-dual}
  Let $M \in \D$, $F \in \Fcal^+$, and $E=F^\cd$. Then the following holds:
  $$\sup k \otimes_R^\mathbf{L} \RHom_R(M,E) = \sup \RHom_R(\RHom_R(k,M),E). $$
\end{prop}
\begin{proof}
  This follows from \cref{tensor-vanishing} using duality. Indeed, we have 
  $$\RHom_R(\RHom_R(k,M),F^\cd) \cong (F \otimes_R^\mathbf{L} \RHom_R(k,M))^\cd$$
  so that
  $$\sup k \otimes_R^\mathbf{L} \RHom_R(M,E) = -\inf F \otimes_R^\mathbf{L} \RHom_R(k,M).$$ 
  On the other hand, 
  $$k \otimes_R^\mathbf{L} \RHom_R(M,F^\cd) \cong k \otimes_R^\mathbf{L} (M \otimes_R^\mathbf{L} F)^\cd$$
  so that
  $$\sup k \otimes_R^\mathbf{L} \RHom_R(M,F^\cd) = \sup k \otimes_R^\mathbf{L} (M \otimes_R^\mathbf{L} F)^\cd = $$
  $$= \sup (M \otimes_R^\mathbf{L} F)^\cd = - \inf (M \otimes_R^\mathbf{L} F) = - \inf \RHom_R(k,M \otimes_R^\mathbf{L} F)$$
  using \cite[Theorem 2.1, Theorem 4.1]{FI03}.
\end{proof}

\bibliographystyle{abstract}
\bibliography{bibitems}
\end{document}